\documentclass[preprint,10.5pt, times]{article}

\usepackage{amssymb}
\usepackage{amsthm}
\usepackage{amsmath}
\usepackage{amscd}
\usepackage{enumerate}

\newtheorem{thm}{Theorem}[section]
\newtheorem{prop}{Proposition}[section]
\newtheorem{lem}{Lemma}[section]
\newtheorem{cor}{Corollary}[section]

\newtheorem{dfn}{Definition}
\newtheorem{exam}{Example}

\newtheorem{rmk}{Remark}

\newcommand{\IE}{\textit{i.e., }}
\newcommand{\CF}{\textit{cf.\ }}	
\newcommand{\EG}{\textit{e.g., }}

\newcommand{\IntegerRing}{ {\mathbb Z} }
\newcommand{\RationalField}{ {\mathbb Q} }
\newcommand{\RealField}{ {\mathbb R} }
\newcommand{\ComplexField}{ {\mathbb C} }
\newcommand{\Idele}[1]{ {{\mathbb I}_{#1}} }
\newcommand{\kreflex}{ K^{*} }
\newcommand{\hreflex}{ H^{*} }
\newcommand{\AbelianExt}[1]{ {{#1}^{ab}} }
\newcommand{\Kernel}[1]{ {{\textrm Ker}\left(#1\right)} }
\newcommand{\Indrp}[2]{ {{\rm Ind}^{#1}_{#2}} }
\newcommand{\IndZtwo}{ {\Indrp{G}{H_0}(\IntegerRing_2)} }
\newcommand{\Galgp}[2]{ {\textrm{Gal}({#1}/{#2})} }
\newcommand{\abs}[1]{ {\left\lvert #1 \right\rvert} }
\newcommand{\Parths}[1]{ {\left(#1\right)} }

\newcommand{\CG}[3]{ {H^{#1}( {#2}, {#3})} }
\newcommand{\Vermap}[1]{ {\textrm{Ver}\left({#1}\right)} }
\newcommand{\ProjLim}{\mathop{\varprojlim}\limits}

\newcommand{\QuadExt}[2]{#1( \sqrt{#2} )}
\newcommand{\FieldExt}[2]{#1( #2 )}
\newcommand{\SetN}{\{ \varphi_1 H_0, \cdots, \varphi_N H_0 \}}
\newcommand{\Jodd}{ {J_{\textit{odd}}} }

\newcommand{\casea}{%
	\begin{cases}
		0  &  \text{if } \tau_1^{-1}\varphi_i\ \in \Phi  \\
		1  &  \text{if } \tau_1^{-1}\varphi_i\ \in \iota \Phi  
	\end{cases}%
}
\newcommand{\caseb}{%
	\begin{cases}
		0  &  \text{if } r_\Phi(\tau_1)(\varphi_i H_0) = 0  \\
		1  &  \text{if } r_\Phi(\tau_1)(\varphi_i H_0) = 1 
	\end{cases}%
}

\begin{document}




\title{On some algebraic properties of CM-types of CM-fields and their reflexes}




\author{Ryoko Oishi-Tomiyasu}

\maketitle

\begin{abstract}
The purpose of this paper is to show that 
the reflex fields of a given CM-field $K$ 
is equipped with a certain combinatorial structure that has not been exploited yet.

The first theorem is on the abelian extension generated by the moduli and the $b$-torsion points of abelian varieties of CM-type, for any natural number $b$.
It is a generalization of the result by Wei on the abelian extension obtained by the moduli and all the torsion points.
The second theorem gives a character identity of the Artin $L$-function of a CM-field $K$ and the reflex fields of $K$.
The character identity pointed out by Shimura in \cite{Shimuraii} follows from this.

The third theorem states that some Pfister form is isomorphic to the orthogonal sum of $Tr_{\kreflex(\Phi)/\RationalField}(\bar{a}a)$ defined on the reflex fields $\bigoplus_{\Phi \in \Lambda} \kreflex(\Phi)$. 
This result suggests that the theory of complex multiplication on abelian varieties has a relationship with the multiplicative forms in higher dimension. 
\end{abstract}





%
%
%
%
\section*{Introduction}
\label{}
The moduli of abelian varieties of a CM-type and their torsion points generate an abelian extension, not of the field of complex multiplication,
but of a reflex field of the field.
This is the theory of Shimura and Taniyama, which determines the Galois group of the abelian extension as the kernel of the half norm map between the idele groups (\CF \cite{Shimurai}, Main Theorem 1, 2).

After that, in order to give a more direct description of the Galois group,
Shimura \cite{Shimuraiii} and Ovseevich \cite{Ovseevich} investigated into the abelian extension generated by the class-field of the maximal totally real subfield
of the CM-field $K$ and complex multiplication.
There is also the work by Kubota on the abelian extension generated by one CM-type of $K$ \cite{Kubota}.
On the other hand, 
Wei proved a theorem on the abelian extension generated only by complex multiplication \cite{Wei}.
Our first theorem (\ref{thm:main}) is a generalization of the theorem by Wei, which deals the abelian extension generated by the moduli and the $b$-torsions for a natural number $b$.
The proof uses a combinatorial property of the reflex fields (Corollary \ref{cor:2^N-1 on idele}).

In section \ref{Character identity}, a set of CM-fields $K(I)$ including $K$ is proposed as a dual set of the reflex fields of $K$.
A character identity holds between these two sets (Theorem \ref{thm:character identity}).
The character identity mentioned in \cite{Shimuraii} is obtained from this.
The set of $K(I)$ consists of only conjugate fields of $K$ over $\RationalField$, if the degree of $K$ is $2$, $4$ or $8$.

The third theorem (\ref{thm:Pfister form and reflex fields}) states the relation between reflex fields and some Pfister form.
Let $K_0$ be the maximal totally real field of a given CM-field $K$. 
Then, $K = \QuadExt{K_0}{-d}$ for some totally positive element $d \in K_0$.
Using the embeddings $\varphi_1, \cdots, \varphi_N : K_0 \hookrightarrow \ComplexField$,
the Pfister form $q := \langle 1, \varphi_1(d) \rangle \otimes \cdots \otimes \langle 1, \varphi_N(d) \rangle$ 
is defined over the Galois closure of $K_0$ over $\RationalField$.
The third theorem states that $q$ is isomorphic to the orthogonal sum of $2^{-N} Tr_{\kreflex(\Phi)/\RationalField}(\bar{a}a)$
defined on the reflex fields $\kreflex(\Phi)$.

In the case of imaginary quadratic fields $\RationalField(\sqrt{-d})$, the Pfister form $\langle 1, d \rangle$
equals the quadratic form on $\RationalField(\sqrt{-d})$ defined by the norm $N_{\RationalField(\sqrt{-d}) / \RationalField}$.
Theorem \ref{thm:Pfister form and reflex fields} is a generalization of this to higher dimension.

Some basic notations and definitions are given in Section \ref{Some algebraic properties of CM-fields and their reflex fields}.
Each following section depends on the content of Section \ref{Some algebraic properties of CM-fields and their reflex fields}.

%
%
\section{Some algebraic properties of CM-fields and their reflex fields} 
\label{Some algebraic properties of CM-fields and their reflex fields}
In this section, several results on the reflex fields and their Galois groups are shown.
They are used to prove the theorems in the remaining sections. 

%
%
\subsection{Basic notations}
\label{}
Here, recall the definitions of CM-types, their dual CM-types;
let $K$ be a CM field and $K_0$ be its maximal totally real field.  Their Galois closures are denoted by $K^c$  and $K_0^c$ respectively.  We denote the Galois groups;
\begin{eqnarray*}
		G := \Galgp{K^c}{\RationalField},
		\hspace{5mm}
		G_0 := \Galgp{K_0^c}{\RationalField},
\end{eqnarray*}
\vspace*{-10mm}
\begin{eqnarray*}
		H := \Galgp{K^c}{K},
		\hspace{5mm}
		H_0 := \Galgp{K^c}{K_0},
		\hspace{5mm}
		C := \Galgp{K^c}{K_0^c},
\end{eqnarray*}
\noindent
where $\RationalField$ is the rational number field.  There is an exact sequence:
\[
	\begin{CD}
		1 @>>> C @>>> G @>>> G_0 @>>> 1.
	\end{CD}
\]

Since $K^c$ is also a CM-field, the complex conjugation $\iota$ is a central element of $G$.
The pair $(K, \Phi)$ is called a CM-type of $K$, when $\Phi$ is a set of the embeddings of $K \hookrightarrow \ComplexField$ over $\RationalField$ such that $\Phi$ and $\iota\Phi$ are disjoint, and every embedding is contained in $\Phi \cup \iota\Phi$.  
Therefore, $\Phi$ and $\Phi \cup \iota\Phi$ correspond one-to-one to the left cosets $G/H_0$, $G/H$ respectively.

For a CM-type $(K, \Phi)$ and a set $S_\Phi := \bigcup_{\varphi \in \Phi} \varphi H$, define a subgroup $\hreflex(\Phi)$ of $G$ by
\begin{eqnarray}
	\hreflex(\Phi) := \{ \sigma \in \Galgp{K^c}{\RationalField} : \sigma S_\Phi = S_\Phi \}.
\end{eqnarray}
Then, $\iota \notin \hreflex(\Phi)$, and $\hreflex(\Phi) S_\Phi = S_\Phi$.
Hence, $S_\Phi = \hreflex(\Phi) \psi_1 \cup \cdots \cup \hreflex(\Phi) \psi_M$, which is a disjoint union of right cosets.
Let $\kreflex(\Phi)$ be the fixed subfield of $K^c$ by $\hreflex(\Phi)$, then, $\kreflex(\Phi)$ is a CM-field.
If we set $\Phi^* := \{ \psi_1^{-1} |_{\kreflex(\Phi)}, \cdots, \psi_M^{-1} |_{\kreflex(\Phi)} \} $, $(\kreflex(\Phi),\Phi^*)$ is also a CM-type.
$(\kreflex(\Phi),\Phi^*)$ is called the dual CM-type of $(K, \Phi)$, and $\kreflex(\Phi)$ is called the reflex field of $(K, \Phi)$.

%
%
\subsection{On the structure of the Galois groups of CM-fields and their reflex fields} 
\label{subsection:On the structure of the Galois groups of CM-fields and their reflex fields}
In this section, we define a $1$-cocycle map 
 and an embedding $\Galgp{K^c}{\RationalField}  \hookrightarrow  \IndZtwo \rtimes G_0$ to see the structure of the Galois group of the reflex fields.
The results in this section are already discribed in \cite{Dodson} in a similar fashion. 
Here, we reformulate the discriptions in \cite{Dodson} to clarify the relation between the structure and the action of $G$ on the CM-types, for the proofs of the following theorems. 

Let $H_0$ act on $\IntegerRing_2$ trivially.
The induced module $\IndZtwo$ can be regarded as a set of maps ${\rm Map}(G/H_0, \IntegerRing_2)$, on which $G$ acts by $\sigma(f)(\tau H_0) = f(\sigma^{-1}\tau H_0)$.
Since $K_0^c$ is the composite field of all the conjugate fields of $K_0$ over $\RationalField$, 
$C = \bigcap_{\tau \in G} \tau H_0 \tau^{-1}$.
Hence, the actions of $C$ on $G/H_0$ and $\IndZtwo$ are trivial.  Therefore, $G_0 = G/C$ also acts on $G/H_0$ and $\IndZtwo$.

As an abelian group, we have $C \cong (\IntegerRing_2)^m$ for some integer $m$.
In fact, since $K = K_0(\sqrt{-d})$ for some totally positive $d \in K_0$, $K^c$ equals $K_0^c( \sqrt{-\varphi_1(d)}, \cdots, \sqrt{-\varphi_N(d)} )$, 
where $\varphi_1, \cdots, \varphi_N$ are the embeddings of $K_0 \hookrightarrow \ComplexField$ over $\RationalField$.

The following lemma shows that 
when $G$ acts on $C$ by conjugation $\tau \cdot a = \tau a \tau^{-1}$ for $\tau \in G$ and $a \in C$,
$C$ can be embedded into $\IndZtwo$ as a $G$-module.

\begin{lem}\label{lem:embedding r}
	Using a totally positive $d$ such that $K = K_0(\sqrt{-d})$, define a map $r : C \longrightarrow  \IndZtwo$ by
	\begin{eqnarray}
		r(a)(\varphi_i H_0) = 
		\begin{cases}
			0  & \text{if } a(\sqrt{-\varphi_i(d)}) = \sqrt{-\varphi_i(d)}, \\
			1  & \text{otherwise.}
		\end{cases}
	\end{eqnarray}
	Then, $r$ does not depend on the choice of $d$, and it is an injective $G$-homomorphism.
\end{lem}

\begin{proof}
	It is clear that $r$ does not depend on the choice of $d$.
	Furthermore, it is an injective homomorphism of additive groups, 
	because $K^c = K_0^c(\sqrt{-\varphi_1(d)}, \cdots, \sqrt{-\varphi_N(d)})$.  
	For any $\tau \in G$, 
	if we put $(-1)^{n_i} = \tau^{-1} \left( \sqrt{-\varphi_i(d)} \right) / \sqrt{-\tau^{-1} \varphi_i(d)}$, then
	\begin{eqnarray*}
		\tau a \tau^{-1} ( \sqrt{-\varphi_i(d)} )
 						&=& (-1)^{n_i} \tau a ( \sqrt{-\tau^{-1}\varphi_i(d)} ) \nonumber \\
						&=& (-1)^{n_i+r(a)(\tau^{-1} \varphi_i H_0)} \tau ( \sqrt{-\tau^{-1}\varphi_i(d)} ) \nonumber \\
						&=& (-1)^{r(a)(\tau^{-1} \varphi_i H_0)} \sqrt{-\varphi_i(d)}.
	\end{eqnarray*}
	So, $(\tau \cdot r)(a)(\varphi_i(d)) := r(a)(\tau^{-1} \varphi_i H_0) = r(\tau a \tau^{-1})(\varphi_i H_0)$. Hence, $r$ is a $G$-homomorphism.  
\end{proof}

The embedding $r$ given in Lemma \ref{lem:embedding r} can be extended to a $1$-cocycle map $G \longrightarrow \IndZtwo$.

\begin{dfn}\label{lem:1-cocycle r}
	For a CM-type $(K, \Phi)$ with $\Phi = \{ \varphi_1, \cdots, \varphi_N \}$, define a map $r_\Phi : G \longrightarrow \IndZtwo$ to be 
	\begin{eqnarray}
		r_\Phi(\tau)(\varphi_i H_0) := 
		\begin{cases}
			0  & \text{if } \tau^{-1}\varphi_{i} \in \Phi  \\
			1  & \text{otherwise.}
		\end{cases}
	\end{eqnarray} 
\end{dfn}


\begin{prop}\label{prop:1-cocycle r}
	The map $r_\Phi$ is a $1$-cocycle that equals $r$ on $C$.
	The cohomology class $[r_\Phi]$ in $\CG{ 1 }{ G_0 }{ \IndZtwo }$ is determined independently from the choice of the CM-type $(K, \Phi)$.
	Moreover, $(K, \Phi) \mapsto r_\Phi$ gives a map from the CM-types of $K$ onto the the cohomology class $[r_\Phi]$.
	For any two CM-types $(K, \Phi)$ and $(K, \Phi^\prime)$, $r_\Phi = r_{\Phi^\prime}$ if and only if $\Phi^\prime = \Phi$ or $\iota\Phi$.
\end{prop}

\begin{proof}
For $a \in C$, $a \varphi_i H = \varphi_i H$ or $\iota \varphi_i H$, since $a \varphi_i H_0 = \varphi_i H_0$.
Furthermore, 
$a \varphi_i H = \varphi_i H$ if and only if $a \text{ fixes } \sqrt{-\varphi_i(d)}$,
since $K_0^c \cdot \varphi_i(K) = K_0^c( \sqrt{-\varphi_i(d)} )$.
Hence, $r_\Phi |_C = r$.

	For any $\tau_1, \tau_2 \in G$, 
	\begin{eqnarray*}
		r_\Phi(\tau_1 \tau_2)(\varphi_i H_0) = 0
			&\Longleftrightarrow& (\tau_1\tau_2)^{-1}\varphi_i \in \Phi  \\
			&\Longleftrightarrow& r_\Phi(\tau_2)(\tau_1^{-1} \varphi_i H_0) = \casea  \\  
			&\Longleftrightarrow& r_\Phi(\tau_2)(\tau_1^{-1} \varphi_i H_0) = \caseb  \\  
			&\Longleftrightarrow& r_\Phi(\tau_1)(\varphi_i H_0) + r_\Phi(\tau_2)(\tau_1^{-1} \varphi_i H_0) = 0 \\
			&\Longleftrightarrow& (r_\Phi(\tau_1) + \tau_1 \cdot r_\Phi(\tau_2))(\varphi_i H_0) = 0.
	\end{eqnarray*}
	So $r_\Phi$ is a $1$-cocycle.
	Next, we shall show the map $\Phi \mapsto r_\Phi$ is onto $[r_\Phi]$;
	using the CM-type $\Phi$ and $f \in \IndZtwo$,
	any CM-types $\Phi^\prime$ are represented by   
	$\{ \iota^{f(\varphi_1 H_0)}\varphi_1, \cdots, \iota^{f(\varphi_N H_0)}\varphi_N \}$.
	Then, $r_{\Phi^\prime}(\tau) = r_\Phi(\tau) + \tau \cdot f - f$, 
	which implies that
	\begin{eqnarray*}
		r_{\Phi^\prime}(\tau) = r_\Phi(\tau) \text{ for all } \tau \in G \Longleftrightarrow f \text{ equals } 0 \text{ or } 1 \text{ constantly.}
	\end{eqnarray*}
	\noindent
	Hence, the last statement also holds.
\end{proof}

Since $r_\Phi |_C = r$ and $G/C = G_0$, $r_\Phi$ induces a $1$-cocycle of $\CG{ 1 }{ G_0 }{ \IndZtwo/r(C) }$.
There is the exact sequence of cohomology groups:
\[
\begin{CD}
	\CG{1}{G_0}{ \IndZtwo/r(C) } @>>> \CG{2}{ G_0 }{ r(C) } @>>> \CG{2}{G_0}{\IndZtwo}.
\end{CD}
\]

\begin{lem}\label{lem:extension C by G0}
	Let  $[s]$ be the image of $[r_\Phi]$ by $\CG{1}{G_0}{ \IndZtwo/r(C) } \longrightarrow  \CG{2}{ G_0 }{ r(C) }$.
	Then, $[s] \in \CG{2}{G_0}{ r(C) }$ corresponds to the exact sequence: $1 \longrightarrow C \longrightarrow G \longrightarrow G_0 \longrightarrow 1$.
\end{lem}
\begin{proof}
	For $\sigma \in G_0$, fix a pre-image of $G \longrightarrow G_0$, and denote it by $\hat{\sigma}$.
	Then, $[s]$ in $\CG{2}{G_0}{r(C)}$ is given by
	\begin{eqnarray*}
		x(\sigma, \tau) = r_\Phi(\hat{\sigma}) - r_\Phi(\widehat{\sigma\tau}) + \hat{\sigma} \cdot r_\Phi(\hat{\tau}).
	\end{eqnarray*}
	Let $a := \hat{\sigma} \hat{\tau} \widehat{\sigma \tau}^{-1}$, then $a \in C$, and we have
	\begin{eqnarray*}
		x(\sigma, \tau) &=& r_\Phi(\hat{\sigma}) - r_\Phi(a^{-1} \hat{\sigma} \hat{\tau}) + \hat{\sigma} \cdot r_\Phi(\hat{\tau}) \nonumber \\
			& = & r_\Phi(\hat{\sigma}) - (-r_\Phi(a) + r_\Phi(\hat{\sigma}) + \hat{\sigma} \cdot r_\Phi(\hat{\tau})) + \hat{\sigma} \cdot r_\Phi(\hat{\tau}) \nonumber \\
			& = & r(a) = r(\hat{\sigma} \hat{\tau} \widehat{\sigma \tau}^{-1}).
	\end{eqnarray*}
	This means $x(\sigma, \tau)$ coincides with the exact sequence.
\end{proof}
By Lemma \ref{lem:extension C by G0}, $[s]$ belongs to the kernel of $H^2(G_0, r(C)) \longrightarrow H^2(G_0, \IndZtwo)$ in the exact sequence.
Therefore, there is a commutative diagram of two exact sequences:
\[
\begin{CD}
	1 @>>> C @>>> G @>>> G_0 @>>> 1 \text{ (exact)} \\
	@.	@VrVV	@VVV	@|	@.	\\
	1 @>>> \IndZtwo @>>> \IndZtwo \rtimes G_0 @>>> G_0 @>>> 1 \text{ (exact)}
\end{CD}
\]

The group operation of the semi-direct group $\IndZtwo \rtimes G_0$ is defined by 
\begin{eqnarray}
	(f, \sigma)(f^\prime, \sigma^\prime) = (f + \sigma \cdot f^\prime, \sigma\sigma^\prime).
\end{eqnarray}
The projection on $\IndZtwo$ of the middle vertical homomorphism equals $r_\Phi$ for a CM-type $(K, \Phi)$.
The corollary follows from the discussion above.

\begin{cor}\label{cor:embedding of G}
	Let $(K, \Phi)$ be a CM-type of $K$.
	Let $\rho$ be the canonical epimorphism $G \twoheadrightarrow G_0$.
	Then, the map $\rho_\Phi : G \longrightarrow \IndZtwo \rtimes G_0$ given by $\tau \mapsto (r_\Phi(\tau), \rho(\tau))$ is an injective homomorphism.
\end{cor}
   
By the definition of $r_\Phi$, we have for any $\tau \in G$,
\begin{eqnarray}
	\tau \in \hreflex(\Phi) \Longleftrightarrow r_\Phi(\tau) \text{ equals } 0 \text{ constantly.}
\end{eqnarray}
Therefore, using the embedding $r_\Phi$ and $\rho$ defined in Corollary \ref{cor:embedding of G},
the subgroups of $G$ can be restated as follows; 
\begin{eqnarray}
	C &=& \left\{ \tau \in G : \rho(\tau) = id \right\}, \label{eq:C} \\ 
	\hreflex(\Phi) &=& \left\{ \tau \in G : r_\Phi(\tau) = 0 \text{ for any } \tau \right\}. \label{eq:hreflex}
\end{eqnarray}
In particular, by the injectivity of $r$ on $C$, $C \cap \hreflex(\Phi) = \{ id \}$ for any $(K, \Phi)$.
On the other hand, $C \subset H$.  Hence, $\hreflex(\Phi)$ can equal $H$ only when $C = \{id, \iota\}$.


%
%
\subsection{The conjugacy among CM-types}
\label{The conjugacy among CM-types over RationalField}
It is said that two CM-types $(K, \Phi)$ and $(K, \Phi^\prime)$ are conjugate if and only if $\Phi^\prime = \tau \Phi$ for some $\tau \in G$.
Corresponding to the action of $G$ on the CM-types of $K$, 
a new action of $G$ on $\IndZtwo$ is introduced in this section.

Fix a CM-type $\Phi_0 := \{ \varphi_1, \cdots, \varphi_N \}$ of $K$. Then, for $f \in \IndZtwo$, a CM-type $\Phi_f$ is given by
\begin{eqnarray}\label{eq:definition of Phi_f}
	\Phi_f := \left\{ \iota^{f(\varphi_1 H_0)} \varphi_1, \cdots, \iota^{f(\varphi_N H_0)} \varphi_N \right\}.
\end{eqnarray}
The set of the CM-types of $K$ has a one-to-one correspondence with $\IndZtwo$ by $f \mapsto \Phi_f$.

Regarding $G$ as a subgroup of $\IndZtwo \rtimes G_0$ by the embedding $\rho_{\Phi_0}$ defined in Corollary \ref{cor:embedding of G},
let $G$ act on $\IndZtwo \rtimes G_0$ by multiplication from the left-hand side.
This yields a new action of $G$ on $\IndZtwo$ given by $\tau * f := r_{\Phi_0}(\tau) + \tau \cdot f$.
This action depends on $\Phi_0$.  We shall show the lemma:

\begin{lem}\label{lem:conjugacy class of CM-types}
	The identification between all the CM-types and $\IndZtwo$ given by $f \mapsto \Phi_f$ 
	coincides with their $G$-structure, \IE $\tau \Phi_f = \Phi_{\tau * f}$ for $\tau \in G$.
	Hence, 
	$\hreflex(\Phi_f)$ coincides with the stabilizer of $f$, \IE $\tau * f = f \Longleftrightarrow \tau \in \hreflex(\Phi_f)$. 
	In addition, 
	two CM-types $(K, \Phi_f)$, $(K, \Phi_{f^\prime})$ are conjugate
	if and only if there exists $\tau \in G$ such that $f^\prime = \tau * f$.
\end{lem}

\begin{proof}
	If we show $\tau \Phi_f = \Phi_{\tau * f}$ for any $\tau \in G$, 
	it is easily seen that the remaining statements hold. 
	For $\tau \in G$ and $\varphi \in \Phi_0$, we have 
	\begin{eqnarray*}
		r_{\Phi_f}(\tau)(\varphi H_0) = 0 \Longleftrightarrow \iota^{f(\varphi H_0)} \varphi \in \tau \Phi_f.
	\end{eqnarray*}
	Hence, 
	\begin{eqnarray*}
		\tau \Phi_f = \left\{ \iota^{(r_{\Phi_f}(\tau)+f)(\varphi_1 H_0)} \varphi_1, \cdots, \iota^{(r_{\Phi_f}(\tau)+f)(\varphi_N H_N)} \varphi_N \right\}.
	\end{eqnarray*}
	Since $r_{\Phi_f}(\tau) = r_{\Phi_0}(\tau) + \tau \cdot f - f$,
	\begin{eqnarray}\label{eq:Phi_f}
		r_{\Phi_f}(\tau) + f = r_{\Phi_0}(\tau) + \tau \cdot f = \tau * f.
	\end{eqnarray}
	Therefore, $\tau \Phi_f = \Phi_{\tau * f}$.  We proved the lemma.
\end{proof} 
The set of the CM-types of $K$ is divided into the orbits by the action of $G$.
Each orbit corresponds to a conjugacy class of the CM-types of $K$.
The fixed subgroup of $\Phi$ is $\hreflex(\Phi)$. Therefore,
the degree of the reflex field $\kreflex(\Phi)$ equals the cardinality of 
the orbit. 
From this, the following lemma follows which is also described in \cite{Shimuraiv};
\begin{lem} \label{lem:the sum of the degrees}
Let $\Lambda$ be a system of representatives for the conjugacy classes of the CM-types of $K$.
Then, the sum of the degrees of the reflex fields $\{ \kreflex(\Phi) \}_{\Phi \in \Lambda}$ equals $2^N$.
\end{lem}
It is used in the proof of Theorem \ref{thm:main}.

%
%
\subsection{On a combinatorial  property of half norm maps}
\label{subsection:half norm maps}
For a CM-type $(K, \Phi)$, a half norm map $K^\times \longrightarrow \kreflex(\Phi)^\times$ is defined by $a \mapsto \prod_{\varphi \in \Phi} \varphi(a)$.
In general, for a given $G$-module $M$,
a half norm map $N_\Phi : M^H \longrightarrow M^{\hreflex(\Phi)}$ is defined by $a \mapsto \sum_{\varphi \in \Phi} \varphi(a)$,
where $M^H$, $M^{\hreflex(\Phi)}$ are the subsets of $M$ consisting of the fixed elements by $H$, $\hreflex(\Phi)$ respectively.
Then, $N_{\Phi}\Parths{a + \iota a}$ equals the norm map $N_{G/H} : a \mapsto \sum_{\sigma \in G/H} \sigma(a)$.

The following proposition shows a characteristic property of half norm maps.

\begin{prop}\label{prop:multiplication by 2^N-1}
	Let $M$ be a $G$-module on which $\iota$ acts as $-1$.  Denote the dual CM-type of $(K, \Phi)$ by $(\kreflex(\Phi), \Phi^{*})$.
	Then, two maps are defined using the half norm maps:
	\begin{eqnarray}
		(N_\Phi)_{\Phi \in \Lambda} : M^H &\longrightarrow& \displaystyle \bigoplus_{\Phi \in \Lambda} M^{\hreflex(\Phi)} \nonumber \\
		a &\mapsto& \left( \sum_{\varphi \in \Phi} \varphi(a) \right)_{\Phi \in \Lambda},
	\end{eqnarray}
	\begin{eqnarray}
		\displaystyle \sum_{\Phi \in \Lambda} N_{\Phi^{*}} : \displaystyle \bigoplus_{\Phi \in \Lambda} M^{\hreflex(\Phi)} &\longrightarrow& M^H  \nonumber \\
		(b_\Phi)_{\varphi \in \Phi} &\mapsto& \displaystyle \sum_{\Phi \in \Lambda} \displaystyle \sum_{\psi \in \Phi^{*}} \psi(b_\Phi),
	\end{eqnarray}
	where $\Lambda$ ranges all the conjugacy classes of the CM-types of $K$.
	Then, the composition of the following maps equals the multiplication by $2^{N-1}$ on $M^H$, where $2N$ is the degree of $K$ over $\RationalField$.
	\[
	\begin{CD}
		M^H
			@>(N_\Phi)_{\Phi \in \Lambda}>>
		\displaystyle \bigoplus_{\Phi \in \Lambda} M^{\hreflex(\Phi)} 
			@>\sum_{\Phi \in \Lambda} N_{\Phi^{*}}>>
		M^H.
	\end{CD}
	\]
	\vspace{-5mm}
\end{prop}
This lemma is generalized to Proposition \ref{prop:multiplication by 2^N-1(2)} in Section \ref{A Pfister form and reflex fields}.
Hence, we omit the proof here.

Let $\Idele{L}$  be the idele group of a number field $L$.
For a natural number $b$, define an open subgroup $U_L((b)) \subset \Idele{L}$ by
\begin{eqnarray}
	U_L((b)) := (L_\infty)_+^\times 
									\times \prod_{ {\mathfrak p} \mid (b)} (1+b{\mathcal O}_{\mathfrak p})
									\times \prod_{{\mathfrak p} \nmid (b)} {\mathcal O}_{\mathfrak p}^\times,
\end{eqnarray}
\noindent
where ${\mathcal O}_{\mathfrak p}$ is the ring of integers of the local field $L_{\mathfrak p}$,
and $(L_\infty)_+^\times$ is the connected open subgroup of $1$ in $(L \otimes_\RationalField \RealField)^\times$.

From Proposition \ref{prop:multiplication by 2^N-1}, the corollary follows immediately. 
It will be used in the proof of the theorem \ref{thm:main}.

\begin{cor} \label{cor:2^N-1 on idele}
	The composition of the following two maps equals the multiplication by $2^{N-1}$ on $\Idele{K} / U_K((b)) \Idele{K_0}$.
	\[
	\begin{CD}
		\Idele{K} / \Idele{K_0} U_K((b)) 
			@>(N_\Phi)_{\Phi \in \Lambda}>>
		\displaystyle \bigoplus_{\Phi \in \Lambda} \Idele{\kreflex(\Phi)} / \Idele{\kreflex_0(\Phi)} U_{\kreflex(\Phi)}((b)) 
			@>\sum_{\Phi \in \Lambda} N_{\Phi^{*}}>>
		\Idele{K} / \Idele{K_0} U_K((b)).
	\end{CD}
	\]
	In particular, the elements that belong to the kernel of $(N_\Phi)_{\Phi \in \Lambda}$, has an order dividing $2^{N-1}$.
\end{cor}

\begin{rmk}
	When $K$ is a Galois extension over $\RationalField$, all the reflex fields are contained in $K$.
	Therefore, in this case, the elements that belong to the kernel of $(N_\Phi)_{\Phi \in \Lambda}$ have an order dividing at most $2$.
\end{rmk}

%
%
\section{The Abelian extension by complex multiplication}
\label{Abelian extensions of fields of CM-type}
In this section, we shall prove the theorem on the abelian extension generated by complex multiplication.

First, we give some definitions; for a number field $L$,
let $C_L := \Idele{L} / L^{\times}$, and denote the connected component of the identity in $C_L$ by $D_L$.
By the main theorem of class field theory,
there is the canonical isomorphism $\phi_L : C_L /D_L \stackrel{\cong}{\longrightarrow} \Galgp{\AbelianExt{L}}{L}$.

For a number field $L$ and an integral ideal ${\mathfrak a}$ of $L$, define an open subgroup $U_L({\mathfrak a}) \subset \Idele{L}$ associated with ${\mathfrak a}$ by
\begin{eqnarray} \label{eq:U_K(b)}
	U_L({\mathfrak a}) := (L_\infty)_+^\times 
									\times \prod_{ {\mathfrak p} \mid {\mathfrak a}} (1+{\mathfrak a}{\mathcal O}_{\mathfrak p})
									\times \prod_{{\mathfrak p} \nmid {\mathfrak a}} {\mathcal O}_{\mathfrak p}^\times,
\end{eqnarray}

\noindent
where ${\mathcal O}_{\mathfrak p}$ is the ring of integers of the local field $L_{\mathfrak p}$,
and $(L_\infty)_+^\times$ is the connected open subgroup in $(L \otimes_\RationalField \RealField)^\times$.
Let $L_{\mathfrak a}$ be the ray class field of $L$ modulo ${\mathfrak a}$, then, $\phi_L$ gives an isomorphism $\Idele{L}/L^\times U_L({\mathfrak a}) \longrightarrow \Galgp{L_{\mathfrak a}}{L}$.

Let $A$ be a polarized abelian variety of CM-type $(\kreflex(\Phi),\Phi^*)$.  
For a CM-type $(K, \Phi)$ and the dual CM-type $(\kreflex(\Phi), \Phi^*)$,
we have a half norm map $K^\times \longrightarrow \kreflex(\Phi)^\times$
defined by $a \mapsto \prod_{\varphi \in \Phi} \varphi(a)$.
It also induces a half norm map $N_\Phi : C_K \longrightarrow C_{\kreflex(\Phi)}$.
By the theory of complex multiplication, the abelian extension over $K$ generated by the moduli and ${\mathfrak a}$-torsions of $A$
 corresponds to the kernel of the map induced by $N_\Phi$:
\begin{eqnarray*}
	C_K
			&\longrightarrow&
	\Idele{\kreflex(\Phi)} / (\kreflex(\Phi))^\times U_{\kreflex(\Phi)}({\mathfrak a}),\\
\vspace{-5mm}
	{\mathfrak a} &\mapsto& \prod_{\varphi \in \Phi} \varphi({\mathfrak a}).
\end{eqnarray*}

In the sequel, let $b$ be a natural number, and $\mathcal{M}_{K, b}$ be the subfield of $\AbelianExt{K}$ obtained by adjoining to $K$,
the moduli and the $b$-torsion points of all the polarized abelian varieties of a CM-type with the reflex field contained in $K$.
By the theory of complex multiplication, ${\mathcal M}_{K,b}$ is contained in $K_{(b)}$,
the ray class field of $K$ modulo $b$.
Furthermore, we have the theorem:

\begin{thm}\label{thm:main}
Let $\{ \kreflex(\Phi) \}_{\Phi \in \Lambda}$ be the set of the reflex fields of all the CM-types of $K$,
and $2^{N_0}$ be the maximum $2$-power dividing the degrees of all the reflex fields of $K$.
(Hence, $1 \leq N_0 \leq N$.)

Then, there exists a subgroup $H(b)$ of $\Galgp{K_{(b)}}{{\mathcal M}_{K,b}}$ satisfying the followings:
	\begin{enumerate}[(i)]
		 \item \label{item:1, thm_main} $2^{N-1}\Galgp{K_{(b)}}{{\mathcal M}_{K,b}} \subset H(b) \subset \Galgp{K_{(b)}}{{\mathcal M}_{K,b}}$, \\
\vspace{-5mm}
		 \item \label{item:2, thm_main} $2^{N_0}H(b) \subset \Vermap{ \Galgp{K_{0, (b)}}{K_0 \cdot \RationalField_{(b)} } } \subset H(b)$,
	\end{enumerate}
	where 
	$K_{0, (b)}$ is the ray class field of $K_0$ modulo $b$,
	and 
	$Ver$ is the Verlagerung map $\Galgp{K_{0, (b)}}{K_0} \longrightarrow \Galgp{{K}_{(b)}}{K}$.

In particular,  for any odd prime $p$, the $p$-component of $\Galgp{K_{(b)}}{{\mathcal M}_{K,b}}$ is isomorphic to 
that of $\Vermap{ \Galgp{K_{0, (b)}}{K_0 \cdot \RationalField_{(b)}} }$.
\end{thm}

\begin{proof} 
For an ideal $(b)$, take $U_K((b))$ as (\ref{eq:U_K(b)}).
Let $V_K((b))$ be the image of $U_K((b))$ in $C_K$.
By class field theory, there is a commutative diagram:
	\[
	\begin{CD}
		C_{K_0} / V_{K_0}((b)) @> \stackrel{\phi_{K_0}}{\cong} >> \Galgp{K_{0, (b)}}{K_0} \\
		@VVV	@V{Ver.}VV	\\
		C_{K} / V_{K}((b)) @> \stackrel{\phi_{K}}{\cong} >> \Galgp{K_{(b)}}{K}
	\end{CD}
	\]
 
To prove the theorem, we use the commutative diagram of two exact sequences:
	\[
\hspace{-25mm}
	\small{
	\begin{CD}
		0 @>>> C_K^{\langle \iota \rangle} / V_K((b))^{\langle \iota \rangle} @>>> C_K / V_K((b)) @>>> C_K / {C_K^{\langle \iota \rangle} V_K((b))} @>>> 0 \\
		@.	@VVV	@V(N_\Phi)_{\Phi \in \Lambda}VV	@VVV	@.\\
		0
			@>>>
		\displaystyle \bigoplus_{\Phi \in \Lambda} C_{\kreflex(\Phi)}^{\langle \iota \rangle} / V_{\kreflex(\Phi)}((b))^{\langle \iota \rangle}
			@>>>
		\displaystyle \bigoplus_{\Phi \in \Lambda} C_{\kreflex(\Phi)} / V_{\kreflex(\Phi)}((b))
			@>>>
		\displaystyle \bigoplus_{\Phi \in \Lambda} C_{\kreflex(\Phi)} / C_{\kreflex(\Phi)}^{\langle \iota \rangle} V_{\kreflex(\Phi)}((b))
		@>>> 0 \\
	\end{CD}
	}
	\]
	\noindent
	where the superscript $\langle \iota \rangle$ means the subgroup fixed by the complex conjugation $\iota$.
	We denote each kernel of the vertical maps by
	\begin{eqnarray*}
		N_{1,b} &:=& \Kernel{ C_K^{\langle \iota \rangle} / V_K((b))^{\langle \iota \rangle}
						\longrightarrow \textstyle \bigoplus_{\Phi \in \Lambda} C_{\kreflex(\Phi)}^{\langle \iota \rangle} / V_{\kreflex(\Phi)}((b))^{\langle \iota \rangle} }, \\
		N_{2,b} &:=& \Kernel{ C_K/V_K((b))
						\longrightarrow \textstyle \bigoplus_{\Phi \in \Lambda} C_{\kreflex(\Phi)} / V_{\kreflex(\Phi)}((b)) }, \\
		N_{3,b} &:=& \Kernel{ C_K / {C_K^{\langle \iota \rangle} V_K((b))}
						\longrightarrow \textstyle \bigoplus_{\Phi \in \Lambda} C_{\kreflex(\Phi)} / C_{\kreflex(\Phi)}^{\langle \iota \rangle} V_{\kreflex(\Phi)}((b)) }.
 	\end{eqnarray*}
	By the snake lemma, we have the exact sequence:
	\[
	\begin{CD}
		0 @>>> N_{1,b} @>>> N_{2,b} @>>> N_{3,b}.
	\end{CD}
	\]
	Furthermore, by the theory of complex multiplication, $\Galgp{K_{(b)}}{{\mathcal M}_{K,b}}$ equals $\phi_K(N_{2,b})$.
	On the other hand, by Corollary \ref{cor:2^N-1 on idele}, all the elements of $N_{3,b}$ have an order dividing $2^{N-1}$.

	Let $H(b) := \phi_K(N_{1,b})$.  
	Then, it is clear that the assertion (\ref {item:1, thm_main}) of the theorem holds.

	Next, we show (\ref {item:2, thm_main});
	since $H^1(K^\times)=0$, we have $C_K^{\langle \iota \rangle} = \Idele{K}^{\langle \iota \rangle} / (K^\times)^{\langle \iota \rangle} = C_{K_0}$.
	Hence, there is a commutative diagram:
	\[
	\begin{CD}
		C_{K_0} / V_{K_0}((b)) @>>> C_K^{\langle \iota \rangle} / V_K((b))^{\langle \iota \rangle} @>>> 0 \text{ (exact) } \\
		@V{N_{K_0/\RationalField}}VV	@V(N_\Phi)_{\Phi \in \Lambda}VV	@.\\
		C_\RationalField / V_\RationalField((b))
			@>>>
		\displaystyle \bigoplus_{\Phi \in \Lambda} C_{\kreflex(\Phi)}^{\langle \iota \rangle} / V_{\kreflex(\Phi)}((b))^{\langle \iota \rangle}
		@. \\
	\end{CD}
	\]
	Let $J_{K_0,b}$ be the kernel of the left vertical map in the diagram, and $J_{K,b}$ be the image of $J_{K_0,b}$ in $C_K / V_K((b))$.
	Then, $\phi_K(J_{K,b}) = \Vermap{ \Galgp{K_{0, (b)}}{K_0 \cdot \RationalField_{(b)} } }$.
	Since $J_{K,b} \subset N_{1,b}$, (\ref {item:2, thm_main}) is proved if $2^{N_0} N_{1,b} \subset J_{K,b}$.

	Let $n_\Phi$ be the degree of $\kreflex(\Phi)$.
	If $a \in C_\RationalField$ belongs to $V_{\kreflex(\Phi)}((b))^{\langle \iota \rangle}$ for all CM-type $(K, \Phi)$,
 	then, $a^{n_\Phi} = N_{G/\hreflex(\Phi)}(a) \in V_\RationalField((b))$.
 	Hence, $a^m \in V_\RationalField((b))$, where $m$ is the greatest common divisor of $n_\Phi$ ($\Phi \in \Lambda$).
%
	Since the sum of $n_\Phi$ ($\Phi \in \Lambda$) equals $2^N$ by Lemma \ref{lem:the sum of the degrees}, $m$ equals $2^{N_0}$ in the statement of the theorem.
\end{proof}

In \cite{Wei}, Wei proved the following theorem on the abelian extension obtained from all the CM-types whose reflex is contained in $K$.
In the remaining part of this section, 
we prove the theorem by Wei,
using Theorem \ref{thm:main} and basic results of algebraic number theory.

\begin{thm}\label{thm:Wei} (Wei)
Let $\mathcal{M}_K$ be the subfield of $\AbelianExt{K}$ obtained by adjoining to $K$, 
the moduli and the torsion points of all the polarized abelian varieties of a CM-type with the reflex field contained in $K$.
The subgroup corresponding to $\mathcal{M}_K$
equals the image of $\Galgp{\AbelianExt{K_0}}{K_0 \cdot \AbelianExt{\RationalField} }$ in $\Galgp{\AbelianExt{K}}{K}$ under the Verlagerung map.
\end{thm}
\begin{proof}
Denote each kernel of the maps induced by $(N_\Phi)_{\Phi \in \Lambda}$, by
\begin{eqnarray*}
	N_1 &:=& \Kernel{ C_K^{\langle \iota \rangle} / D_K^{\langle \iota \rangle}
					\longrightarrow \textstyle \bigoplus_{\Phi \in \Lambda} C_{\kreflex(\Phi)}^{\langle \iota \rangle} / D_{\kreflex(\Phi)}^{\langle \iota \rangle} }, \\
	N_2 &:=& \Kernel{ C_K/D_K
						\longrightarrow \textstyle \bigoplus_{\Phi \in \Lambda} C_{\kreflex(\Phi)} / D_{\kreflex(\Phi)} }, \\
	N_3 &:=& \Kernel{ C_K / {C_K^{\langle \iota \rangle} D_K}
						\longrightarrow \textstyle \bigoplus_{\Phi \in \Lambda} C_{\kreflex(\Phi)} / C_{\kreflex(\Phi)}^{\langle \iota \rangle} D_{\kreflex(\Phi)} }.
\end{eqnarray*}
These groups equal
	$\displaystyle \ProjLim_b N_{1,b}$, 
	$\displaystyle \ProjLim_b N_{2,b}$, 
	$\displaystyle \ProjLim_b N_{3,b}$ respectively,
the inverse limits of the groups defined in the proof of Theorem \ref{thm:main}.
In addition, we define
\begin{eqnarray*}
	J_{K_0} := \Kernel{ N_{K_0/\RationalField} : C_{K_0} / D_{K_0}
						\longrightarrow C_\RationalField / D_\RationalField }.
\end{eqnarray*}
and denote the image in $C_K / D_K$ by $J_K$.
$J_{K_0}$ is also the inverse limit of $J_{K_0, b}$ in the proof of Theorem \ref{thm:main}.
Then, Theorem \ref{thm:Wei} follows from Lemma \ref{lem:5.1} and Lemma \ref{lem:5.2}.
\end{proof}

\begin{lem}\label{lem:5.1}
	$N_1=J_K$.
\end{lem}

\begin{lem}\label{lem:5.2}
	$N_3 = \{0 \}$.
\end{lem}

Lemma \ref{lem:5.5} is used to prove them.
First, we prove Lemma \ref{lem:5.5} using Lemma \ref{lem:5.4}.

Let $E_L$ be the group of units of a number field $L$.
For any subgroup $U \subset \Idele{L}$, $\overline{U}$ means its topological closure.
The image of $L_\infty^\times \subset \mathbb{I}_L$ in $C_L$ is denoted by $\widetilde{L}_\infty^\times$.
.


\begin{lem}\label{lem:5.4}
	Let $L$ be a number field.
	If $a \in \Idele{L}$ satisfies $a^m \in \overline{L^\times}$ (resp. $a^m \in \overline{ L_\infty^\times L^\times }$) for a natural number $m$, 
	there exists $d \in \overline{L^\times}$ such that $a^m = d^m$ (resp. $a_{\mathfrak p}^m = d_{\mathfrak p}^m$ at any finite primes ${\mathfrak p}$).
  	In particular, if $a \in C_L$ satisfies $a^m \in D_L$ (resp. $a \in \widetilde{L}_\infty^\times D_L$), then, 
	$a$ corresponds to $b \in \Idele{L}$
 	such that all the components $b_\nu$ at any primes (resp. any finite primes) $\nu$ are $m$-th roots of unity.
\end{lem}

\begin{proof}
	Let $U_m \subset \Idele{L}$ be an open subgroup such that $E_L \cap U_m \subset E_L^m$.
	Take $c \in L^\times$ such that $a^m \equiv c$ mod $U_m$.
	For any open subgroup $U^\prime \subset U_m$,
	there exists a unit $\epsilon$ of $L$ such that $a^m \equiv \epsilon^m c$ mod $U^\prime$.
	Therefore, for any finite prime ${\mathfrak p}$ of $L$, $c^{1/m} \in L_{\mathfrak p}$. 
	This means that the primes of $L(\zeta_m)$ except for a finite number,
	are totally split over $L(c^{1/m}, \zeta_m)$, where $\zeta_m$ is a primitive $m$-th root of unity.
	Hence, $c^{1/m}$ belongs to $L(\zeta_m)$.
	Then, $c^{1/m} \in L$, again since we have $c^{1/m} \in L_{\mathfrak p}$ for any finite prime ${\mathfrak p}$. Therefore, the first statement follows.
	(Suppose that $U_m, U^\prime \supset L_\infty^\times$ for the case of $a^m \in \overline{ L_\infty^\times L^\times }$.)
 	The second statement follows immediately.
\end{proof}

\begin{lem}\label{lem:5.5}
	Let  $L \subset M$ be number fields. 
	If $a \in C_L$ satisfies $a \in D_M$ (resp. $a \in \widetilde{M}_\infty^\times D_M$) and $a^m \in D_L$ for a natural number $m$,  
	then, $a \in D_L$ (resp. $a \in \widetilde{L}_\infty^\times D_L$).
\end{lem}

\begin{proof}
	By Lemma \ref{lem:5.4}, we may assume that the components of $a$ at any primes are $m$-th roots of unity.
	Let $U$ be an open subgroup in $\Idele{M}$ such that no roots of unity of $M$ are contained in it.  
	Take a positive number $l$ such that $E_M^l \subset U \cap E_M$.
	For any open subgroup $U^\prime \subset U$ such that $U^\prime \cap E_M \subset E_M^{ml}$, there exists $\epsilon \in E_M$ such that $a \equiv \epsilon$ mod $U^\prime$, because $a \in D_M$.
	Then, $\epsilon^m \in E_M^{ml}$,
	since $\epsilon^m \equiv a^m \equiv 1$ mod $U^\prime$.
	Therefore, $\zeta \epsilon \in E_M^l$ for some root of unity $\zeta \in M$.
	In particular, $a \equiv \epsilon \equiv \zeta$ mod $U$.
	Since $U$ is arbitrary, $a = \zeta$ for some root of unity $\zeta$ in $L$. Therefore, $a \in D_L$.
	The proof is similar for the case of $a \in \widetilde{M}_\infty^\times D_M$.
\end{proof}

\begin{proof}[Proof of Lemma \ref{lem:5.1}]
	Since $D_K^{\langle \iota \rangle} = (\widetilde{K_0})_\infty^\times D_{K_0}$, 
	we have $C_K^{\langle \iota \rangle}/D_K^{\langle \iota \rangle} = C_{K_0} / (\widetilde{K_0})_\infty^\times D_{K_0}$. 
	Therefore,
	\begin{eqnarray*}
		N_1 = \Kernel{ C_{K_0} / (\widetilde{K_0})_\infty^\times D_{K_0} \longrightarrow \textstyle \bigoplus_{\Phi \in \Lambda} C_{\kreflex_0(\Phi)} / (\widetilde{\kreflex_0(\Phi)})_\infty^\times D_{\kreflex_0(\Phi)} }.
	\end{eqnarray*}
	This map is via the norm map $N_{K_0/\RationalField} : C_{K_0} / \widetilde{(K_0)}_\infty^\times D_{K_0} \longrightarrow C_\RationalField / \widetilde{\RealField}^\times D_\RationalField$. 
	Take an element $a \in C_\RationalField$ that belongs to $ (\widetilde{\kreflex_0(\Phi)})_\infty^\times D_{\kreflex_0(\Phi)}$.  
	Then, $a^2 \in D_{\kreflex_0(\Phi)}$.
	In addition, 
	the canonical map $C_\RationalField / D_\RationalField \longrightarrow C_{\kreflex_0(\Phi)} / D_{\kreflex_0(\Phi)}$ is injective,
	because both $\kreflex_0(\Phi)$ and $\RationalField$ are totally real fields \cite{Artin}.
	Hence, $a^2 \in D_\RationalField$.  
	Therefore, $a \in \widetilde{\RealField}^\times D_\RationalField$ by Lemma \ref{lem:5.5};
	the statement follows.
\end{proof}

\begin{proof} [Proof of Lemma \ref{lem:5.2}]
	Take $a \in C_K$ such that $N_\Phi(a) \in C_{\kreflex(\Phi)}^{\langle \iota \rangle} D_{\kreflex(\Phi)}$ for all the CM-types $(K, \Phi)$.
	It is enough if we can show $a \in C_K^{\langle \iota \rangle} D_K$, which is equivalent to $a^{1-\iota} \in D_K$. 
	By Theorem $\ref{thm:main}$, $a^{2^{N-1}} \in C_K^{\langle \iota \rangle} D_K$.  Hence, $a^{2^{N-1}(1-\iota)} \in D_K$.
	By Lemma \ref{lem:5.5}, there exists $b \in \Idele{K}$ such that $b \equiv a^{1-\iota}$ mod $D_K$, 
	and all the components of $b$ at the primes of $K$, are $2^{N-1}$-th roots of unity.  
	Moreover, $N_\Phi(b)$ is a root of unity of $\kreflex(\Phi)$ for any $\Phi$, since $N_\Phi(b) \in D_{\kreflex(\Phi)}$.
	Hence, $b^{1-\iota} = b^2$ is a root of unity, and $b$ is also a root of unity.
	Therefore, $a^{1-\iota} \in D_K$.  
\end{proof}

%
%
\section{A Character identity}
\label{Character identity}
In this section, we give a proof of a character identity between the Artin $L$-functions of a CM-field $K$ and the reflex of $K$.

For any group $N_0$ and a normal subgroup $N$ such that $[ N_0 : N ] = 2$, 
we denote by $\chi_{N_0/N}$ the non-trivial character of $N_0$ 
induced by the canonical homomorphisms $N_0 \twoheadrightarrow N_0/N \stackrel{\cong}{\to} \{{\pm 1}\}$.

\begin{prop} 
	Let $K$ be a CM-field.  Let ${\mathbf 0}$ and ${\mathbf 1}$ be the elements of $\IndZtwo$ that map any $\varphi_i H_0$ to $0$ and $1$ respectively.
	For a fixed CM-type $(K, \Phi_0)$, $G$ can be regarded as a subgroup of $\IndZtwo \rtimes G_0$ by $\rho_{\Phi_0}$ defined in Corollary \ref{cor:embedding of G}.
	Then, the equation holds:
	\begin{eqnarray} 
	{\rm Res}_G\ \Indrp{ \IndZtwo \rtimes G_0}{\langle{\mathbf 1}\rangle \times G_0}
	(\chi_{\langle{\mathbf 1}\rangle \times G_0/\{{\mathbf 0}\}\times G_0})  
	= \displaystyle \sum_{\Phi \in \Lambda} \Indrp{G}{\hreflex_0(\Phi)}(\chi_{\hreflex_0(\Phi)/\hreflex(\Phi)}),
	\end{eqnarray} 
	\noindent
	where $\langle{\mathbf 1}\rangle$ is the subgroup of $\IndZtwo$ generated by ${\mathbf 1}$,
	and $\Lambda$ is a system of representatives for the conjugacy classes of the CM-types of $K$.
\end{prop}

\begin{proof}
	We have 
	\begin{eqnarray*} 
		{\rm Res}_G\ \Indrp{\IndZtwo \rtimes G_0}{\langle{\mathbf 1}\rangle \times G_0}
		(\chi_{\langle{\mathbf 1}\rangle\times G_0/\{{\mathbf 0}\}\times G_0})  
			= 
		\displaystyle \sum_{s \in G 
			\backslash \IndZtwo \rtimes G_0/  
			\langle{\mathbf 1}\rangle \times G_0} 
			\Indrp{G}{s(\langle{\mathbf 1}\rangle \times G_0)s^{-1} \cap G}(\chi^s), 
	\end{eqnarray*} 
	\noindent
	where $\chi^s$ is the character of $s(\langle{\mathbf 1}\rangle\times G_0)s^{-1} \cap G$
	such that $\chi^s(x) = \chi_{\langle{\mathbf 1}\rangle \times G_0/\{{\mathbf 0}\}\times G_0}(s^{-1}xs)$  
	for $x \in s(\langle{\mathbf 1}\rangle\times G_0)s^{-1} \cap G$.   
	Let $\{s_i\}$ be a system of representatives for
	$G \backslash \IndZtwo \rtimes G_0/ \langle{\mathbf 1}\rangle \times G_0$.
	We can choose $\{s_i\}$ from $\IndZtwo \times \{id\}$.
	Let $s_i = (f_i, id)$, and 
	$\Phi_{f_i} := \{ \iota^{f_i(\varphi_1 H_0)} \varphi_1, \cdots, \iota^{f_i(\varphi_N H_0)} \varphi_N \}$ for a fixed CM-type $\Phi_0 := \{ \varphi_1, \cdots, \varphi_N \}$.
	Then, by Lemma \ref{lem:conjugacy class of CM-types}, $\{ \Phi_{f_i} \}$ makes a system of representatives for the conjugacy class of the CM-types of $K$.
	Moreover, we have
	\begin{eqnarray*}
		s_i(\langle{\mathbf 1}\rangle\times G_0)s_i^{-1} \cap G 
			= \{ \sigma \in G : \sigma * f_i = f_i \text{ or } f_i + {\mathbf 1} \}
			= \hreflex_0(\Phi_{f_i}),
	\end{eqnarray*} 
	\noindent
	Hence, $\chi^{s_i}$ is the character induced by $\hreflex_0(\Phi_{f_i}) \twoheadrightarrow \hreflex_0(\Phi_{f_i})/\hreflex(\Phi_{f_i}) \stackrel{\cong}{\to} \{{\pm 1}\}$. 
	By this, the statement follows.
\end{proof}

Let $I$ be a subset of $\SetN$. 
Then, the character of $\IndZtwo$ is given by 
$\chi_{I}(f) :=(-1)^{ \sum_{\varphi H_0 \in I} f(\varphi H_0)}$.
It is clear that $I \mapsto \chi_{I}$ gives a bijection from the subsets of 
$\SetN$ to the characters of $\IndZtwo$.  
For $f \in \IndZtwo$ and $s = (f^\prime, \sigma) \in \IndZtwo \rtimes G_0$, 
we define $f^s$ such that $(f^s, id) = s^{-1}\cdot(f, id )\cdot s$
, and $\chi_{I}^s(f) := \chi_{I}(f^s)$.  
Then, for $s = (f^\prime, \sigma)$,
\begin{eqnarray}
	\chi_{I}^{s}(f)
		&=& \chi_{I}( \sigma^{-1}\cdot f + \sigma^{-1}\cdot f^\prime - f^\prime ) \nonumber \\
		&=& \chi_{I}( \sigma^{-1}\cdot f + \sigma^{-1}\cdot f^\prime ) - \chi_{I}( f^\prime ) \nonumber \\ 
		&=& \chi_{ \sigma I }( f + f^\prime ) - \chi_{I}( f^\prime ),
\end{eqnarray}
\noindent
where $\sigma I := \{ \sigma \varphi_1 H_0, \cdots, \sigma \varphi_N H_0 \}$.
Hence, $\chi_{I}^s = \chi_{I}$ means $\chi_{ \sigma I }( f + f^\prime ) = \chi_{I}( f + f^\prime )$ for any $f$,
therefore, $\sigma I = I$.

Let $H_0(I) := \{ \sigma \in G_0 : \sigma I = I \}$. Then, 
$\chi_{I}$ is extended to a character of $\IndZtwo \rtimes H_0(I)$ by 
\begin{eqnarray}
	\chi_{I}((f, \sigma)) = \chi_{I}(f).
\end{eqnarray} 
This character is denoted by $\tilde{\chi_{I}}$.

For an irreducible representation $\pi$ of $H_0(I)$,  
let $\tilde\pi$ be the irreducible representation of $\IndZtwo \rtimes H_0(I)$ 
given by composing $\pi$ with $\IndZtwo \rtimes H_0(I) \twoheadrightarrow H_0(I)$.
We denote the representation $\Indrp{\IndZtwo \rtimes G_0}{\IndZtwo \rtimes H_0(I)}(\tilde{\chi_{I}}\otimes\tilde{\pi})$
by $\theta_{I, \pi}$ and the character of $\theta_{I, \pi}$ by $\chi_{I, \pi}$.

\begin{prop} 
Let $\abs{I}$ be the cardinality of $I$.  Then,
	\begin{enumerate}[(a)]
		\item $\theta_{I, \pi}$ is irreducible.

		\item If $\theta_{I, \pi}$ and $\theta_{I^\prime, \pi^\prime}$ are isomorphic,
			then $I^\prime = \sigma I$ for some $\sigma \in G_0$, and $\pi^\prime$ is isomorphic to $\pi$.

		\item Every irreducible representation of $\IndZtwo \rtimes G_0$ is isomorphic to one of the $\theta_{I, \pi}$.  
	\end{enumerate}
\end{prop}

\begin{proof}
	See, \EG \cite{Serreii}, Proposition 25.
\end{proof}

For any odd number $d$, 
$G_0$ acts on $\SetN$ by $\varphi_i H_0 \mapsto \sigma \varphi_i H_0$,
which induces an action of $G_0$ on the set $\{ I \subset \SetN : \abs{I} = d \}$ for any natural number $d$.
Let $J_d$ be a system of representatives for the orbits,
and $\Jodd := \bigcup_{1 \leq d \leq N : \text{ odd }} J_d$.

\begin{lem}
	The following gives an irreducible decomposition:
	\begin{eqnarray}
		\Indrp{ \IndZtwo \rtimes G_0}{\langle{\mathbf 1}\rangle \times G_0}(\chi_{\langle{\mathbf 1}\rangle \times G_0/\{{\mathbf 0}\}\times G_0})
			= \sum_{I \in \Jodd} \chi_{I,id},
	\end{eqnarray}
\end{lem}

\begin{proof}
In the sequel, for any representation $\pi$, we denote the character of $\pi$ by $\chi_{\pi}$. 
From the above proposition, it suffices if the multiplicities of $\theta_{I, \pi}$ are computed.
By the Frobenius reciprocity,
	\begin{eqnarray*}
		&& \hspace{-10mm}
			\langle \Indrp{ \IndZtwo \rtimes G_0 }{\langle{\mathbf 1} \rangle \times G_0}
				( \chi_{
					\langle{\mathbf 1}\rangle \times G_0 \slash \{ {\mathbf 0} \} \times G_0
				} ),\ \chi_{I,\pi}
			\rangle_{\IndZtwo \rtimes G_0} \\  
		&=& \langle
				\chi_{\langle{\mathbf 1}\rangle \times G_0 \slash \{{\mathbf 0}\} \times G_0},\
				{\rm Res}_{\langle{\mathbf 1}\rangle \times G_0} \Indrp{\IndZtwo \rtimes G_0}{\IndZtwo \rtimes H_0(I)}
				( \tilde{\chi_{I}}\otimes\chi_{\tilde{\pi}})
			\rangle_{\langle{\mathbf 1}\rangle \times G_0} \\ 
		&=& \langle
				\chi_{\langle{\mathbf 1}\rangle \times G_0 \slash \{{\mathbf 0}\}\times G_0},\
				\Indrp{\langle{\mathbf 1}\rangle \times G_0}{\langle{\mathbf 1}\rangle \times H_0(I)}
				\left(
					\tilde{\chi_{I}}\otimes\chi_{\tilde{\pi}}
				\right)
			\rangle_{\langle{\mathbf 1}\rangle \times G_0 } \\             
		&=& \langle
			\chi_{\langle{\mathbf 1}\rangle \times H_0(I) \slash \{{\mathbf 0}\}\times H_0(I)},\
			\tilde{\chi_{I}}\otimes\chi_{\tilde{\pi}}
			\rangle_{\langle{\mathbf 1}\rangle \times H_0(I)}
	\end{eqnarray*}
	Moreover,
	\begin{eqnarray*}
		& & \hspace{-10mm}
			\langle
				\chi_{\langle{\mathbf 1}\rangle \times H_0(I) \slash \{{\mathbf 0}\}\times H_0(I)},\
				\tilde{\chi_{I}}\otimes\chi_{\tilde{\pi}}
			\rangle_{\langle{\mathbf 1}\rangle \times H_0(I)} \\
		&=& \abs{ \langle{\mathbf 1}\rangle \times H_0(I) }^{-1} 
				\displaystyle \sum_{t \in \langle {\mathbf 1} \rangle \times H_0(I)}
			\chi_{\langle{\mathbf 1}\rangle \times H_0(I)
			\slash \{{\mathbf 0}\}\times H_0(I)}(t)
			(
				\tilde{\chi_{I}}\otimes\chi_{\tilde{\pi}}
			)(t^{-1})  \\
		&=& \frac{1}{2} \abs{ H_0(I) }^{-1} 
			\displaystyle \sum_{t \in \{{\mathbf 0}\}\times H_0(I)}
				\left(
					\tilde{\chi_{I}}\otimes\chi_{\tilde{\pi}}
				\right)
				( t^{-1} ) 
				( 1 - (\tilde{\chi_{I}}\otimes\chi_{\tilde{\pi}})(({\mathbf 1}, id )^{-1}))
			\\
		&=& \frac{1-\chi_{I}({\mathbf 1})}{2} \abs{ H_0(I) }^{-1}
			\displaystyle \sum_{t \in \{{\mathbf 0}\} \times H_0(I)}
			\tilde{\chi_{I}} \otimes \chi_{\tilde{\pi}} (t^{-1}) \\
		&=& \displaystyle \frac{1-\chi_{I}({\mathbf 1})}{2} \langle \pi,\ id \rangle_{H_0(I)} 
	\end{eqnarray*}

	Therefore,
	\begin{eqnarray*}
		&& \hspace{-15mm} 
		\langle 
			\Indrp{ \IndZtwo \rtimes G_0}{\langle{\mathbf 1} \rangle \times G_0}
			( \chi_{\langle{\mathbf 1}\rangle \times G_0 \slash \{{\mathbf 0}\}\times G_0}),\ \chi_{I,\pi}
		\rangle_{\IndZtwo \rtimes G_0}
		 \\
		&=& 
		\begin{cases}
			1 & \text{ if } \abs{I} \text{ is odd and } \pi \text{ is the trivial representation of } G_0,  \\
			0 & \text{ otherwise. }
		\end{cases}
	\end{eqnarray*}
	We obtain the lemma.
\end{proof}   

Let $I$ be a subset of $\SetN$ such that $\abs{I}$ is odd.
Define
\begin{eqnarray}
	H(I) := \left\{ \sigma \in G : \textstyle \sum_{\varphi H_0 \in I} r_\Phi(\sigma)(\varphi H_0)=0,\ \rho(\sigma) \in H_0(I) \right\}.
\end{eqnarray}
The definition of $H(I)$ is independent of the choice of $\Phi$.
If $I=\{ H_0 \}$, then, $H(I) = H$.  Therefore $K(I) = K$.
Moreover,
\begin{eqnarray}
	{\rm Res}_H(\chi_{I,id}) 
		= {\rm Res}_G \Indrp{\IndZtwo \rtimes G_0}{\IndZtwo \rtimes H_0(I)}(\tilde{\chi_{I}})
		= \Indrp{G}{H(I)}(\tilde{\chi_I}),
\end{eqnarray}
\noindent
since $\abs{ G \backslash ( \IndZtwo \rtimes G_0 ) \slash ( \IndZtwo \rtimes H_0(I) ) } = 1$.

We denote the field corresponding to the subgroup $( \IndZtwo \rtimes H_0(I) ) \cap G$ by $K_0(I)$.  
Since $\iota \in ( \IndZtwo \rtimes H_0(I) ) \cap G$, $K_0(I)$ is a totally real field. 
Let $K(I)$ be the quadratic extension of $K_0(I)$ corresponding to the kernel of the character $\tilde{\chi_{I}}$.
Since $\iota$ does not belong to the kernel, $K(I)$ is a CM-field. 

We obtain the theorem from the above argument.
\begin{thm} \label{thm:character identity}
	Let $\Lambda$ be a system of representatives for 
	the conjugacy classes of the CM-types of $K$. 
	Then,
	\begin{eqnarray} \label{eq:character identity}
		\displaystyle \sum_{(K, \Phi) \in \Lambda} \Indrp{G}{\hreflex_0(\Phi)}(\chi_{\hreflex_0(\Phi)/\hreflex(\Phi)}) 
			= \sum_{I \in \Jodd}
				\Indrp{G}{H(I)}(\tilde{\chi_{I}}),
	\end{eqnarray}
	\noindent
\end{thm}

Now, we can apply the character identity 
to obtain a relation formula of the Artin $L$-functions and relative invariants of a CM-field and its reflexes,
as is done in \cite{Shimuraii}.

For a field $F$, we denote the discriminant, the class number 
and the unit group of $F$ by $d_F$,  $h_F$ and $E_F$ respectively.

\begin{cor} 
There is a relation formula of the Artin $L$-functions:  
	\begin{eqnarray}
		\displaystyle \prod_{(K, \Phi) \in \Lambda} L(s, \chi_{\hreflex_0(\Phi)/\hreflex(\Phi)}) 
			= \prod_{I \in \Jodd} L(s, \tilde{\chi_{I}}), 
	\end{eqnarray}
	\noindent
	where $\Lambda$ is a system of representatives for the conjugacy classes of the CM-types of $K$.    
	Therefore,
	\begin{eqnarray}
		\displaystyle \prod_{(K, \Phi) \in \Lambda} 
			\abs{ \frac{d_{\kreflex(\Phi)}}{d_{\kreflex_0(\Phi)}} }
		&=&
				\prod_{I \in \Jodd} \abs{ \frac{d_{K(I)}}{d_{K_0(I)}} } \\
		\displaystyle \prod_{(K, \Phi) \in \Lambda} \frac{h_{\kreflex(\Phi)}/h_{\kreflex_0(\Phi)}}{[E_{\kreflex(\Phi)} : E_{\kreflex_0(\Phi)}]} 
		&=&
				\prod_{I \in \Jodd} \frac{h_{K(I)}/h_{K_0(I)}}{[E_{K(I)} : E_{K_0(I)}]}
	\end{eqnarray}  
\end{cor} 

\begin{proof}
The first equation follows from Theorem \ref{thm:character identity}; in general, $L(s, \chi) =L(s, \Indrp{G}{H}(\chi))$ holds for any character $\chi$ of a finite dimensional representaion of $H$.
We denote  
the number of the roots of unity in $K$,  $K_0$ by $w_K$, $w_{K_0}$,  
and the regulator of $K$,  $K_0$ by $R_K$, $R_{K_0}$.
Then,
\begin{eqnarray*}
	L(1, \chi_{K/K_0}) 
		&=&
			\lim_{s \to 1}\frac{\zeta_K(s)}{\zeta_{K_0}(s)}
		=
			\left( \frac{(2\pi)^N h_K R_K}{w_K \sqrt{\abs{d_K}}} \right) 
				\left( \frac{2^N h_{K_0} R_{K_0}}{w_{K_0} \sqrt{ \abs{d_{K_0}} }} \right)^{-1}    \\
		&=&
			\pi^N \abs{ \frac{d_K}{d_{K_0}} }^{-\frac{1}{2}} \frac{h_K/h_{K_0}}{[E_K : E_{K_0}]}. 
\end{eqnarray*}

For a character $\chi$ of $G$,  let $\mathfrak{f}(\chi)$ be the conductor 
of the character $\chi$.  
Then, 
\begin{eqnarray*}
	\left( d_K \right) &=& \mathfrak{f}(\Indrp{G}{H}(id_H))  \\
	\left( d_{K_0} \right) &=& \mathfrak{f}( \Indrp{G}{H_0}(id_{H_0}) ),
\end{eqnarray*}
where $id_H$ and  $id_{H_0}$ are the trivial representation of 
$H$ and $H_0$ respectively (\CF \cite{Serrei}, Chapter VI, Corollary 1 of Prop. 6).
Hence,
\begin{eqnarray*}
	\left( \frac{d_K}{d_{K_0}} \right)
		= \mathfrak{f}(\Indrp{G}{H}(id_H)) \mathfrak{f}(\Indrp{G}{H_0}(id_{H_0}))^{-1} 
		= \mathfrak{f}(\Indrp{G}{H}(\chi_{H_0/H}))\text{.}
\end{eqnarray*}
Therefore, the corollary follows.
\end{proof}

	It is clear that $C \subset H(I)$.
	Therefore, $K(I)$ is different from any reflex fields of $K$ at least when $C \neq \{id,  \rho\}$,
	since $C \cap \hreflex(\Phi) = \{ id \}$ by (\ref{eq:C}) and (\ref{eq:hreflex}) in Section \ref{subsection:On the structure of the Galois groups of CM-fields and their reflex fields}.
	By this, Example \ref{exam:G = iota times G_0} is the only case when the set of $K(I)$ totally coincides with the set of the reflex fields of $K$,
	\IE the character identity is trivial.


\begin{exam}\label{exam:G = iota times G_0}
Suppose that $G = \langle\iota\rangle \times G_0$ and the degree $N$ of $K_0$ is an odd number, \IE there is a section homomorphism $G_0 \hookrightarrow G$.
Let $\Phi_0$ be a CM-type of $K$ such that $S_{\Phi_0} = \{ id \} \times G_0$.
For a subset $I \subset \SetN$, define $f_I \in \IndZtwo$ by
\begin{eqnarray}\label{eq:definition of Phi_I}
	 f_I(\varphi H_0) :=
	\begin{cases}
			0 & \varphi H_0 \notin I, \\
			1 & \varphi H_0 \in I.
	\end{cases}
\end{eqnarray}
Then, for the fixed CM-type $\Phi_0$, we have
$(id, \sigma) * f_I = \sigma \cdot f_I = f_{\sigma I}$, and $(\iota, \sigma) * f_I = 1 + f_{\sigma I}$.
Therefore, when the cardinality of $I$ is odd, we have
\begin{eqnarray*}
\hreflex(\Phi_{f_I}) = \left\{ (id, \sigma) \in \langle\iota\rangle \times G_0 : \sigma I =I  \right\}  = H(I).
\end{eqnarray*}
Hence, $K(I)$ and $\kreflex(\Phi_{f_I})$ coincide.
The set of $K(I)$ ($I \in \Jodd$) has a one-to-one correspondence
with the reflexes $\kreflex(\Phi_{f_I})$ of $K$.
\end{exam}





\begin{exam}
Let $G$ be  the dihedral group $D_{2n}$ of degree $2n$. 
Then, $G$ has two generators $\alpha$,  $\beta$ 
such that $\alpha^{2n}=id$, $\beta^2=id$, and $\beta\alpha\beta^{-1}=\alpha^{-1}$.
Regard the central element $\alpha^n$ as the complex conjugation $\iota$.
Then, the fixed subfield $K$ by $H := \{id, \beta\}$ is a CM-field
and 
$\Phi_0 := \{ id |_H, \alpha |_H, \cdots, \alpha^{n-1} |_H \}$
is a CM-type of $K$.
We have $\hreflex(\Phi_0)=\{ id,  \alpha^{n-1}\beta \}$.
When $n$ is odd, it is the case of Example \ref{exam:G = iota times G_0}.

So, let $n$ be even.
In this case, $K$ and $\kreflex(\Phi_0)$ are not conjugate.
Nevertheless, as Shimura pointed out in \cite{Shimuraii}, 
\begin{eqnarray}\label{eq:shimura}
\Indrp{G}{H_0}(\chi_{H_0/H}) 
	= \Indrp{G}{\hreflex_0(\Phi_0)}(\chi_{\hreflex_0(\Phi_0)/\hreflex(\Phi_0)}).
\end{eqnarray}
In \cite{Dodson}, this type of character identity is investigated for dihedral groups, responding to Shimura's suggestion.
We shall show that the equation (\ref{eq:shimura}) follows from Theorem \ref{thm:character identity};
regard $G_0$ as a subgroup of the symmetry group $S_n$ by the canonical action on $G /H_0 = \{ H_0, \alpha H_0, \cdots, \alpha^{n-1} H_0 \}$. 
A map $\rho_{\Phi_0} := (r_{\Phi_0}, \rho)$ is defined in Corollary \ref{cor:embedding of G}.
The image of $\alpha^i$ is given by 
\begin{eqnarray*} 
r_{\Phi_0}(\alpha^i) = (\underbrace{ 1, \cdots, 1}_i, \underbrace{ 0, \cdots, 0 }_{n-i}),\ 
\rho(\alpha^i) = (1\ 2\ ...\ n)^i.
\end{eqnarray*}
Similarly, the image of $\alpha^i \beta$ is given by 
\begin{eqnarray*} 
r_{\Phi_0}(\alpha^i \beta) &=& (\underbrace{ 0, \cdots, 0}_{i+1}, \underbrace{ 1, \cdots, 1 }_{n-i-1}), \\
\rho(\alpha^i \beta) &=& (\lfloor \frac{n+i+1}{2} \rfloor\ \lceil \frac{n+i+3}{2} \rceil) \cdots (i+3\ n-1)(i+2\ n) (1\ i+1) (2\ i) \cdots ( \lfloor \frac{i+1}{2} \rfloor\ \lceil \frac{i+3}{2} \rceil ),
\end{eqnarray*}
where $\lfloor * \rfloor$ and $\lceil * \rceil$ are the floor function and the ceiling function respectively.

Let $k_0$ be the maximum number such that $2^{k_0} \mid n$.
For the CM-type $\Phi_0$, if we define $\Phi_f$ as (\ref{eq:definition of Phi_f}), 
$\sigma \in G$ is contained in $\hreflex(\Phi_f)$ if and only if $r_{\Phi_0}(\sigma) = \rho(\sigma) \cdot f - f$.
Therefore, 
let $(i, 2n)$ be the greatest common divisor of $i$ and $2n$, then,
\begin{eqnarray*} 
\alpha^i \in \hreflex(\Phi) \text{ for some } \Phi &\Longleftrightarrow & \alpha^{(i, 2 n)} \in \hreflex(\Phi) \text{ for some } \Phi \Longleftrightarrow 2^{k_0+1} \mid i, \\
\alpha^i \beta\ \in \hreflex(\Phi) \text{ for some } \Phi &\Longleftrightarrow& i \text{ is odd. }
\end{eqnarray*}
Similarly, for $I \subset \{ H_0, \alpha H_0, \cdots, \alpha^{n-1} H_0 \}$,
$\sigma \in G$ is contained in $H(I)$ if and only if $\sigma I = I$ and $\sum_{\alpha^j H_0 \in I} r_{\Phi_0}(\sigma)(\alpha^j H_0) = 0$. 
Therefore, under the assumption that the cardinality of $I$ is odd,
\begin{eqnarray*} 
	\alpha^i \in H(I) &\Longleftrightarrow & \alpha^{(i, 2 n)} \in H(I) \Longleftrightarrow 2^{k_0+1} \mid i \text{ and } \alpha^{(i, n)} I = I, \\
	\alpha^i \beta\ \in H(I) &\Longleftrightarrow& i \text{ is even and } \alpha^i \beta I = I. 
\end{eqnarray*}
Hence, $\hreflex(\Phi)$ is conjugate to 
i) $\langle \alpha^{2^{k_0+1}j} \rangle$, or
ii) $\langle \alpha^{2^{k_0+1}j}, \alpha^{n-1}\beta \rangle$.
for some $j$ such that $2^{k_0} j \mid n$.
Similarly, $H(I)$ is conjugate to 
iii) $\langle \alpha^{2^{k_0+1} j} \rangle$, or
iv) $\langle \alpha^{2^{k_0+1} j}, \beta \rangle$. 

In the sequel, $\Phi_0$ is replaced by the CM-type such that $\alpha^i|_H \in \Phi_0 \Longleftrightarrow i \equiv 0, \cdots, 2^{k_0}-1 \text{ mod } 2^{k_0+1}$. 
For $I \subset \{ H_0, \alpha H_0, \cdots, \alpha^{n-1} H_0 \}$, define $f_I \in \IndZtwo$ as (\ref{eq:definition of Phi_I}).
Then, since we have $r_{\Phi_0}(\alpha^{2^{k_0+1}j})=0$ for any $j$, and $r_{\Phi_0}(\alpha^{n-1}\beta)=0$,
\begin{eqnarray*}
\alpha^{2^{k_0+1}j} \in \hreflex(\Phi_{f_I}) &\Longleftrightarrow& \alpha^{2^{k_0}j} I = I,  \\
\alpha^{n-1}\beta \in \hreflex(\Phi_{f_I}) &\Longleftrightarrow& \alpha^{n-1}\beta I = I.
\end{eqnarray*}

For a natural number $j$ such that $2^{k_0} j \mid n$, define 
\begin{eqnarray*}
S_j &:=& \left\{I \subset \SetN : \alpha^{2^{k_0} j} I = I \text{ and } \alpha^{2^{k_0} j_0} I \neq I \text{ for any } j_0 \mid j \right\}, \\
\tilde{ S }_j &:=& \left\{I \in S_j : \alpha^{n-1}\beta I = I \right\}, \\
T_j &:=& \left\{I \in S_j : \abs{ I } \text{ is odd} \right\}, \\
\tilde{ T }_j &:=& \left\{I \in S_j : \abs{ I } \text{ is odd}, \text{ and } \beta I = I \right\},
\end{eqnarray*}
Let $s_j:=\abs{ S_j }$,  $\tilde{s}_j:=\abs{ \tilde{ S }_j }$, $t_j:=\abs{ T_j }$,  $\tilde{t}_j := \abs{ \tilde{T}_j }$.
Using these variables, we can count the number of the orbits of $G$ such that the stabilizer is conjugate to each subgroup.
By Theorem \ref{thm:character identity},
\begin{eqnarray*}
	& &
	\hspace{-10mm}
\displaystyle \sum_{0 < j\ \mid\ 2^{-k_0}n} 
	\frac{ s_j - [G : \langle \alpha^{2^{k_0}j}, \alpha^{n-1}\beta \rangle ] \tilde{s}_j }{ [G : \langle \alpha^{2^{k_0+1}j} \rangle ] }
	\Indrp{G}{\langle \alpha^{2^{k_0}j} \rangle} ( \chi_{\langle \alpha^{2^{k_0}j} \rangle / \langle \alpha^{2^{k_0+1}j} \rangle} ) \\
	& &
 +
 \displaystyle \sum_{0 < j\ \mid\ 2^{-k_0}n} 
	\frac{ \tilde{s}_j }{ [\langle \alpha^{2^{k_0}j}, \alpha^{n-1}\beta \rangle : \langle \alpha^{2^{k_0+1}j}, \alpha^{n-1}\beta \rangle ] } 
	\Indrp{G}{\langle \alpha^{2^{k_0}j},\ \alpha^{n-1}\beta \rangle} ( \chi_{\langle \alpha^{2^{k_0}j},\ \alpha^{n-1}\beta \rangle / \langle \alpha^{2^{k_0+1}j},\ \alpha^{n-1}\beta \rangle} ) \\
 &=& 
\displaystyle \sum_{0 < j\ \mid\ 2^{-k_0}n} 
	\frac{ t_j - [G : \langle \alpha^{2^{k_0-1}j}, \beta \rangle ] \tilde{t}_j }{ [G : \langle \alpha^{2^{k_0}j} \rangle ] }
	\Indrp{G}{\langle \alpha^{2^{k_0}j} \rangle} ( \chi_{\langle \alpha^{2^{k_0}j} \rangle / \langle \alpha^{2^{k_0+1}j} \rangle} ) \\
	& &
 +
\displaystyle \sum_{0 < j\ \mid\ 2^{-k_0}n} 
\frac{ \tilde{t}_j}{ [\langle \alpha^{2^{k_0 -1}j}, \beta \rangle : \langle \alpha^{2^{k_0}j}, \beta \rangle ] } 
\Indrp{G}{\langle \alpha^{2^{k_0}j},\ \beta \rangle} ( \chi_{\langle \alpha^{2^{k_0}j},\ \beta \rangle / \langle \alpha^{2^{k_0+1}j},\ \beta \rangle}).
\end{eqnarray*}
Since we have
$\rho(\alpha) = (1\ 2\ \cdots\ n)$,
$\rho(\beta) = (\frac{n}{2}\ \frac{n+4}{2}) \cdots (3\ n-1)(2\ n)$, and
$\rho(\alpha^{n-1} \beta) = (1\ n) (2\ n-1) \cdots (\frac{n}{2}\ \frac{n+2}{2})$,
the subsets can be restated;
\begin{eqnarray*}
S_j &=& \left\{I \subset \{ 1, \cdots, n \} : 2^{k_0} j \text{ is the minimum periodicity of } I \right\}, \\
\tilde{ S }_j &=& \left\{I \in S_j : i \in I \Leftrightarrow n - i + 1 \in I \right\}, \\
T_j &=& \left\{I \in S_j : \abs{ I \cap \{1, \cdots, 2^{k_0} j\} } \text{ is odd} \right\}, \\
\tilde{ T }_j &=& \left\{I \in S_j : 1 \in I \Leftrightarrow \frac{n}{2} + 1 \notin I, \text{ and } i \in I \Leftrightarrow  n - i + 2 \in I\ (2 \leq i \leq \frac{n}{2}) \right\},
\end{eqnarray*}
Hence, we have $s_j=2t_j$, $\tilde{s}_j = \tilde{t}_j$. Therefore,
\begin{eqnarray*}
 & & \hspace{-20mm}
\displaystyle \sum_{0 < j\ \mid\ 2^{-k_0}n} 
\frac{ \tilde{s}_j }{2} 
\Indrp{G}{\langle \alpha^{2^{k_0}j},\ \alpha^{n-1}\beta \rangle} ( \chi_{\langle \alpha^{2^{k_0}j},\ \alpha^{n-1}\beta \rangle / \langle \alpha^{2^{k_0+1}j},\ \alpha^{n-1}\beta \rangle} ) \\
 &=& 
\displaystyle \sum_{0 < j\ \mid\ 2^{-k_0}n} 
\frac{ \tilde{s}_j}{2} 
\Indrp{G}{\langle \alpha^{2^{k_0}j},\ \beta \rangle} ( \chi_{\langle \alpha^{2^{k_0}j},\ \beta \rangle / \langle \alpha^{2^{k_0+1}j},\ \beta \rangle} ).
\end{eqnarray*}

By induction on $j$,  we obtain (\ref{eq:shimura}). 
\end{exam}


%
%
\section{Reflex fields and a Pfister form}
\label{A Pfister form and reflex fields}
In this section, we show Proposition \ref{prop:multiplication by 2^N-1(2)},
then, give a proof of the third theorem which states that some Pfister form is decomposed into quadratic forms defined on a set of reflex fields of $K$.
Pfister forms are known as the only case of anisotropic multiplicative quadratic forms (\cite{Pfister}, \CF \cite{Lam})

For a CM-type $(K, \Phi)$ and $I \subset \SetN$ with odd cardinality,
a CM-field $K(I)$ is introduced in Section \ref{Character identity}.
These CM-fields are also used for the proof of the third theorem.
First, we give a definition of the CM-type $(K(I), \Phi(I))$;
using $\rho_\Phi : G \hookrightarrow \IndZtwo \rtimes G_0$ in Corollary \ref{cor:embedding of G}, we defined the subgroups $H(I)$ and $H_0(I)$ of $G$ by
\begin{eqnarray}
	H(I) &:=& \left\{ \sigma \in G : \textstyle \sum_{\varphi H_0 \in I}  r_\Phi(\sigma)(\varphi H_0) = 0,\ \rho(\sigma)I = I  \right\}, \label{eq:H(I)} \\
	H_0(I) &:=& \left\{ \sigma \in G : \rho(\sigma)I = I  \right\}. \label{eq:H_0(I)}
\end{eqnarray}
Then, $H(I)$ and $H_0(I)$ are independent of the choice of the CM-type $\Phi$, and $H_0(I) = H(I) \cup \iota H(I)$.
A CM-field $K(I)$ is the fixed subfield by $H(I)$.
Furthermore, if we set
\begin{eqnarray}\label{eq:S_Phi(I)} 
	S_{\Phi(I)} := \left\{  \sigma \in G : \textstyle \sum_{\varphi H_0 \in I } r_\Phi(\sigma^{-1})(\varphi H_0) = 0 \right\},
\end{eqnarray}
then, $\hreflex(\Phi) S_{\Phi(I)} = S_{\Phi(I)}$ and $S_{\Phi(I)} H(I)  = S_{\Phi(I)}$.
Denote the set of the embeddings $K(I) \hookrightarrow \ComplexField$ corresponding to the left cosets $S_{\Phi(I)} / H(I)$ by $\Phi(I)$,
and the set of the embeddings $\kreflex(\Phi) \hookrightarrow \ComplexField$ given by the inverse of the right cosets $\hreflex(\Phi) \backslash S_{\Phi(I)}$ by $\Phi(I)^*$.
$\Phi(I)$ and $\Phi(I)^*$ are CM-types of $K(I)$ and $\kreflex(\Phi)$ respectively.
In addition, $(K(I), \Phi(I))$ contains the dual CM-type of $( \kreflex(\Phi), \Phi(I)^*)$, and 
$(\kreflex(\Phi), \Phi(I)^*)$ contains the dual of $( K(I), \Phi(I) )$.

For $0 \leq d \leq N$, there is the canonical action of $G$ on $\{ I \subset \SetN : \abs{I} = d \}$;
let $J_d$ be a system of representatives for the orbits, and $\Jodd := \bigcup_{1 \leq d \leq N : \text{ odd }} J_d$.
Since the stabilizer of $I$ is $\langle H(I), \iota \rangle$,
there is the decomposition given by $H_0(I) \sigma \mapsto \sigma^{-1} I$:
\begin{eqnarray}\label{eq:decomposition1}
	\left\{ I \subset \SetN : \abs{I} \text{ is odd } \right\} = \bigcup_{I \in \Jodd} H_0(I) \backslash G.
\end{eqnarray}

On the other hand, for a fixed CM-type $(K, \Phi_0)$, a $G$-action on $\IndZtwo$ is defined by $f \mapsto \sigma * f$
 in Section \ref{The conjugacy among CM-types over RationalField}.
Since the stabilizer of $f$ is $\hreflex(\Phi_f)$,
there is the decomposition for each $I$, given by $\sigma \hreflex_0(\Phi_f) \mapsto \sigma * f$:
\begin{eqnarray}\label{eq:decomposition2}
	\IndZtwo / \langle {\mathbf 1} \rangle = \bigcup_{\Phi_f \in \Lambda} G / \hreflex_0(\Phi_f),
\end{eqnarray}
where $\langle{\mathbf 1}\rangle$ is the subgroup of $\IndZtwo$ generated by ${\mathbf 1} \in \IndZtwo$ that maps all $G/H_0$ to $1$,
and $\Lambda$ is a system of representatives for the conjugacy classes of the CM-types of $K$.
When the degree of $K$ is $2N$,
the dimension of $\bigoplus_{I \in \Jodd} K(I)$ over $\RationalField$ is $2^N$,
which is same as that of
$\bigoplus_{\Phi \in \Lambda} \kreflex(\Phi)$.

The following proposition also shows that the set of $(K(I), \Phi(I))$ can be regarded as the dual of the set of $(\kreflex(\Phi), \Phi^*)$;
for a $G$-module $M$, denote by $M^H$, the subset of $M$ consisting of all the fixed elements by $H$.
Then, the half norm map $N_{\Phi(I)} : M^{H(I)} \longrightarrow M^{\hreflex(\Phi)}$ is defined by $a \mapsto \sum_{\varphi \in \Phi(I)} \varphi(a)$.
The norm map $N_{G/H} : M^H \longrightarrow M^G$ is also defined by
$a \mapsto \sum_{\sigma \in G/H} \sigma(a)$.

\begin{prop}\label{prop:multiplication by 2^N-1(2)}
	Let $M$ be a $G$-module on which $\iota$ acts as $-1$.
	Using the half norm maps, define two maps $N_{J \rightarrow \Lambda}$, $N_{\Lambda \rightarrow J}$ by
	\begin{eqnarray*}
		N_{J \rightarrow \Lambda} :
		\displaystyle \bigoplus_{I \in \Jodd} M^{H(I)}
			&\longrightarrow&
		\displaystyle \bigoplus_{\Phi_f \in \Lambda} M^{\hreflex(\Phi_f)} \nonumber \\
		(a_I)_{I \in \Jodd}
			&\mapsto&
		\left( \sum_{I \in \Jodd}  \iota^{\sum_{\varphi H_0 \in I} f(\varphi H_0)} N_{\Phi_f(I)}(a_I) \right)_{\Phi_f \in \Lambda}, \\
	\end{eqnarray*}
	\begin{eqnarray*}
		N_{\Lambda \rightarrow J} :
		\displaystyle \bigoplus_{\Phi_f \in \Lambda} M^{\hreflex(\Phi_f)} 
			&\longrightarrow& 
		\displaystyle \bigoplus_{I \in J_{\text{odd}}} M^{H(I)}  \nonumber \\
		(b_{\Phi_f})_{\varphi \in \Phi_f}
			&\mapsto&
		\left( \displaystyle \sum_{\Phi_f \in \Lambda} \iota^{\sum_{\varphi H_0 \in I} f(\varphi H_0)}  N_{\Phi_f(I)^*} (b_{\Phi_f}) \right)_{I \in J_{\text{odd}}}.
	\end{eqnarray*}
	Then, the compositions $N_{J \rightarrow \Lambda} \circ N_{\Lambda \rightarrow J}$ and 
	$N_{\Lambda \rightarrow J} \circ N_{J \rightarrow \Lambda}$ equal the multiplication by $2^{N-1}$.
\end{prop}

We use the following two lemmas for the proof.

\begin{lem}\label{lem:eq1}
	For any $I, I^{\prime} \in \Jodd$,
	\begin{eqnarray}\label{eq:lemma_eq1}
				& & \hspace{-20mm}
			\displaystyle \sum_{\Phi_f \in \Lambda} 
					\iota^{ \sum_{\varphi^{\prime} H_0 \in I^{\prime}} f(\varphi^{\prime} H_0) + \sum_{\varphi H_0 \in I} f(\varphi H_0) } 
						N_{ \Phi_f(I^{\prime})^* } \circ N_{\Phi_f(I)} \nonumber \\
				&=&
				\begin{cases}
					2^{N-2} (id - \iota) + 2^{N-2} N_{G/H(I)}  & \text{if } I = I^\prime, \\
					0  & \text{otherwise.}
				\end{cases}
	\end{eqnarray}
\end{lem}
\begin{proof}
	The left-hand side of (\ref{eq:lemma_eq1}) equals
	\begin{eqnarray*}
			& & 
			\hspace{-10mm}
			\displaystyle \sum_{\Phi_f \in \Lambda}
						\displaystyle \sum_{\psi^\prime \in G/\hreflex_0(\Phi_f) }
							\iota^{\sum_{\psi H_0 \in I^{\prime}} (f + r_{\Phi_f}(\psi^\prime))(\varphi H_0)} \psi^\prime 
								\displaystyle \sum_{\psi \in H_0(I) \backslash G} \iota^{\sum_{\varphi H_0 \in I} (f + r_{\Phi_f}(\psi))(\varphi H_0)} \psi^{-1} \\
			&=& 
			\displaystyle \sum_{\Phi_f \in \Lambda}
						\displaystyle \sum_{\psi^\prime \in G/\hreflex_0(\Phi_f) }
							\iota^{\sum_{\varphi H_0 \in I^{\prime}} \psi^\prime * f(\varphi H_0)} 
								\displaystyle \sum_{\psi \in H_0(I) \backslash G} \iota^{\sum_{\varphi H_0 \in I} \psi * f (\varphi H_0)} \psi^\prime \psi^{-1} \\
			&=&
			\displaystyle \sum_{\Phi_f \in \Lambda}
						\displaystyle \sum_{\psi^\prime \in G/\hreflex_0(\Phi_f) }
							\iota^{\sum_{\varphi H_0 \in I^{\prime}} \psi^\prime * f(\varphi H_0)} 
								\displaystyle \sum_{\psi \in H_0(I) \backslash G} \iota^{\sum_{\varphi H_0 \in I} \psi \psi^\prime * f (\varphi H_0)} \psi^{-1} \\
			&=:& (\star).
	\end{eqnarray*}
	Using the decomposition (\ref{eq:decomposition2}),
	\begin{eqnarray*}
			(\star)
			&=& 
			\displaystyle \sum_{ f \in \IndZtwo/ \langle {\mathbf 1} \rangle }
							\iota^{\sum_{\varphi H_0 \in I^{\prime}} f(\varphi H_0)} 
								\displaystyle \sum_{\psi \in H_0(I) \backslash G} \iota^{\sum_{\varphi H_0 \in I} \psi * f (\varphi H_0)} \psi^{-1} \\
			&=&
			\displaystyle \sum_{\psi \in H_0(I) \backslash G} \iota^{\sum_{\varphi H_0 \in I} r_{\Phi_0}(\psi)(\varphi H_0)} \psi^{-1}
			\displaystyle \sum_{ f \in \IndZtwo/ \langle {\mathbf 1} \rangle } 
							\iota^{\sum_{\varphi H_0 \in I^{\prime}} f(\varphi H_0) + \sum_{\varphi H_0 \in I} \psi \cdot f (\varphi H_0)} \\
			&=&
			\displaystyle \sum_{\psi \in H_0(I) \backslash G} \iota^{\sum_{\varphi H_0 \in I} r_{\Phi_0}(\psi)(\varphi H_0)} \psi^{-1}
			\displaystyle \sum_{ f \in \IndZtwo/ \langle {\mathbf 1} \rangle } 
							\iota^{\sum_{\varphi H_0 \in I^{\prime}} f(\varphi H_0) + \sum_{\varphi H_0 \in \psi^{-1} I} f (\varphi H_0)} \\
			&=&
					2^{N-1} \displaystyle \sum_{\psi \in H_0(I) \backslash G,\ I^\prime = \psi^{-1} I } \psi^{-1}
					+ 2^{N-2} \displaystyle \sum_{\psi \in H_0(I) \backslash G,\ I^\prime \neq \psi^{-1} I } (1 + \iota ) \psi^{-1}.
	\end{eqnarray*}
	If $I^{\prime} = \psi^{-1} I$ for some $\psi  \in H_0(I) \backslash G$, then, $I = I^\prime$ and $\psi \in H_0(I)$ since $I$, $I^{\prime} \in \Jodd$.  Therefore, 
	\begin{eqnarray*}
			(\star)
			=
				\begin{cases}
					2^{N-2} (id - \iota) + 2^{N-2} N_{G/H(I)} & \text{if } I = I^\prime, \\
					0  & \text{otherwise.}
				\end{cases}
	\end{eqnarray*}
	We obtain the lemma. 
\end{proof}

\begin{lem}\label{lem:eq2}
	For $\Phi_{f}, \Phi_{f^{\prime}} \in \Lambda$, 
	\begin{eqnarray}\label{eq:lemma_eq2}
			\displaystyle \sum_{I \in \Jodd} \iota^{\sum_{\varphi H_0 \in I} (f^{\prime} + f)(\varphi H_0)}
						N_{ \Phi_{f^{\prime}}(I) } \circ N_{\Phi_f(I)^* }
				=
				\begin{cases}
					2^{N-2} (id - \iota) + 2^{N-2} N_{G/\hreflex(\Phi_f)} & \text{if } f = f^{\prime},  \\
					0  & \text{otherwise.}
				\end{cases}%
	\end{eqnarray}
\end{lem}

\begin{proof}
	The left-hand side of (\ref{eq:lemma_eq2}) equals
	\begin{eqnarray*}
		& &
		\hspace{-10mm}
			\displaystyle \sum_{I \in \Jodd}
						\displaystyle \sum_{\psi^\prime \in H_0(I) \backslash G}
							\iota^{\sum_{\varphi H_0 \in I} (f + f^\prime + r_{\Phi_{f^\prime}}(\psi^\prime))(\varphi H_0)} {\psi^\prime}^{-1} 
								\displaystyle \sum_{\psi \in G / \hreflex_0(\Phi_f) } \iota^{\sum_{\varphi H_0 \in I} r_{\Phi_f}(\psi)(\varphi H_0)}  \psi \\
			&=&
			\displaystyle \sum_{I \in \Jodd}
						\displaystyle \sum_{\psi^\prime \in H_0(I) \backslash G}
							\iota^{\sum_{\varphi H_0 \in I} (f+ f^\prime + r_{\Phi_{f^\prime}}(\psi^\prime))(\varphi H_0)} 
								\displaystyle \sum_{\psi \in G / \hreflex_0(\Phi_f) } \iota^{\sum_{\varphi H_0 \in I} r_{\Phi_f}(\psi^\prime \psi)(\varphi H_0)}  \psi \\
			&=:& (\star).
	\end{eqnarray*}
 	Since $r_{\Phi_{f^{\prime}} }(\tau) = r_{ \Phi_{f} }(\tau) + \tau \cdot (f - f^{\prime}) - (f - f^{\prime})$, 
	\begin{eqnarray*}
		(f+ f^\prime + r_{\Phi_{f^\prime} }(\psi^\prime) + r_{\Phi_f}(\psi^\prime \psi))(\varphi H_0)
			&=& (\psi^\prime \cdot( f - f^\prime ) + r_{\Phi_f }(\psi^\prime) + r_{\Phi_f}(\psi^\prime \psi))(\varphi H_0) \\
			&=& (\psi^\prime \cdot( f - f^\prime ) + \psi^\prime \cdot r_{\Phi_f}(\psi))(\varphi H_0) \\
			&=& \psi^\prime \cdot( \psi * f - f^\prime)(\varphi H_0).
	\end{eqnarray*}
	Then, by the decomposition (\ref{eq:decomposition1}),
	\begin{eqnarray*}
		(\star)
			&=&
			\displaystyle \sum_{I \in \Jodd}
						\displaystyle \sum_{\psi^\prime \in H_0(I) \backslash G}
								\displaystyle \sum_{\psi \in G / \hreflex_0(\Phi_f) } \iota^{\sum_{\varphi H_0 \in {\psi^\prime}^{-1} I} ( \psi * f - f^\prime)(\varphi H_0)}  \psi \\
			&=&
				\displaystyle \sum_{\psi \in G / \hreflex_0(\Phi_f) } \psi
				\displaystyle \sum_{I \subset \SetN,\ \abs{I} \text{ : odd } }
						\iota^{\sum_{\varphi H_0 \in I} ( \psi * f - f^\prime)(\varphi H_0)} \\
			&=& 
					2^{N-1}
				\displaystyle \sum_{\psi \in G / \hreflex_0(\Phi_f),\ f^\prime = \psi * f } \psi
					+ 2^{N-1} \displaystyle \sum_{\psi \in G / \hreflex_0(\Phi_f),\ f^\prime = \iota\psi * f } \iota \psi \\
					& & + 2^{N-2} \displaystyle \sum_{\psi \in G / \hreflex_0(\Phi_f),\ f^\prime = \iota\psi * f } (id + \iota) \psi.
	\end{eqnarray*}
	If $\psi * f = f^\prime$ for some $\psi  \in G / \hreflex(\Phi_f)$, then, $f^\prime  = f$ and $\psi  \in \hreflex(\Phi_f)$ since $f$, $f^{\prime} \in \Lambda$.  Therefore, 
	\begin{eqnarray*}
		(\star)
			&=& 
				\begin{cases}
					2^{N-2} (id - \iota) + 2^{N-2} N_{G/\hreflex(\Phi_f) } & \text{if }  f = f^{\prime}, \\
					0  & \text{otherwise.}
				\end{cases}%
	\end{eqnarray*}
	We obtain the lemma. 
\end{proof}

\begin{proof}[Proof of Proposition \ref{prop:multiplication by 2^N-1(2)}]
	Denote the canonical embedding $M^{H(I^\prime)} \hookrightarrow \bigoplus_{I \in \Jodd} M^{H(I)}$ by $i_{I^\prime}$,
	and the canonical projection $\bigoplus_{I \in \Jodd} M^{H(I)} \twoheadrightarrow M^{H(I^\prime)}$
	by $p_{I^\prime}$.
	Then, for $I^\prime, I^{\prime\prime} \in \Jodd$ by Lemma \ref{lem:eq1},
	\begin{eqnarray*}
		p_{I^{\prime\prime}} \circ N_{\Lambda \rightarrow J} \circ N_{J \rightarrow \Lambda} \circ i_{I^\prime} 
			&=& \displaystyle \sum_{\Phi_f \in \Lambda}
				\iota^{\sum_{\varphi^{\prime\prime} H_0 \in I^{\prime\prime}} f(\varphi^{\prime\prime} H_0) + \sum_{\varphi^\prime H_0 \in I^\prime} f(\varphi^\prime H_0)} 
						N_{\Phi_f(I^{\prime\prime})^*} \circ N_{\Phi_f(I^\prime)} \\
			&=& 
					\begin{cases}
						2^{N-1} id & \text{if } I^{\prime\prime} = I^\prime,  \\
						0  &  \text{otherwise.}
					\end{cases}%
	\end{eqnarray*}

	Similarly, denote the canonical embedding $M^{\hreflex(\Phi^\prime)} \hookrightarrow  \bigoplus_{\Phi \in \Lambda} M^{\hreflex(\Phi)}$
	and the projection $\bigoplus_{\Phi \in \Lambda} M^{\hreflex(\Phi)} \twoheadrightarrow M^{\hreflex(\Phi^\prime)}$
	by $i_{\Phi^\prime}$ and $p_{\Phi^\prime}$.
	Then, for $\Phi_{f^\prime}, \Phi_{f^{\prime\prime}} \in \Lambda$ by Lemma \ref{lem:eq2}, 
	\begin{eqnarray*}
		p_{\Phi_{f^{\prime\prime}}} \circ N_{J \rightarrow \Lambda} \circ N_{\Lambda \rightarrow J} \circ i_{\Phi_{f^\prime}} 
			&=& \displaystyle \sum_{I \in \Jodd} \iota^{\sum_{\varphi H_0 \in I} (f^{\prime} + f^{\prime\prime})(\varphi H_0)}
						N_{\Phi_{f^{\prime\prime}}(I)} \circ
							N_{\Phi_{f^{\prime}}(I)^* } \nonumber \\
			\hspace{10mm}
			&=&
				\begin{cases}
					2^{N-1} id & \text{if } f^\prime = f^{\prime\prime},  \\
					0 & \text{otherwise.}
				\end{cases}%
	\end{eqnarray*}
	Therefore, the assertion holds.
\end{proof}

For a given CM-field $L$, we denote by $L^{1-\iota}$, the subspace consisting of $a \in L$ such that $\iota a = -a$.
Then, $N_{\Lambda \rightarrow J}$ defined in Proposition \ref{prop:multiplication by 2^N-1(2)} induces an isomorphism as $\RationalField$-linear spaces:
		\begin{eqnarray}\label{eq:eq3}
			N_{\Lambda \rightarrow J} : \displaystyle \bigoplus_{\Phi \in \Lambda} \kreflex(\Phi)^{1-\iota}
				\stackrel{\cong}{\longrightarrow }
			\displaystyle \bigoplus_{I \in \Jodd} K(I)^{1-\iota}.
		\end{eqnarray}

Even if the cardinality of $I$ is even, $H(I)$, $H_0(I)$ and $S_{\Phi(I)}$ are well-defined by (\ref{eq:H(I)}), (\ref{eq:H_0(I)}) and (\ref{eq:S_Phi(I)}).
$\Phi(I)$ and $\Phi(I)^*$ are also defined similarly. 
When the cardinality of $I$ is odd, we have $\iota \in H(I)$, hence, the fixed field $K(I)$ by $H(I)$ is a totally real field.
$[H_0(I) : H(I)] = 1$ or $2$ in this case.

Let $M$ be a $2$-divisible $G$-module $M$ on which $\iota$ acts as $-1$.
Then, for $I \subset \SetN$ with odd cardinality, the half norm map $N_{\Phi(I)}$, $N_{\Phi(I)^*}$ satisfies on $M^{H(I)}$, $M^{\hreflex(\Phi)}$ respectively,
\begin{eqnarray}
	N_{\Phi(I)} &=& \frac{1}{2} \sum_{ \psi \in H(I) \backslash G } (-1)^{\sum_{\varphi H_0 \in I} r_{\Phi}(\psi)(\varphi H_0)} \psi^{-1}, \label{eq:def2_N_{Phi(I)}} \\
	N_{\Phi(I)^*} &=&\frac{1}{2} \sum_{\psi \in G / \hreflex( \Phi ) } (-1)^{\sum_{\varphi H_0 \in I} r_{\Phi}(\psi)(\varphi H_0)} \psi \label{eq:def2_N_{Phi(I)^*}}.
\end{eqnarray}
In the sequel, we use (\ref{eq:def2_N_{Phi(I)}}), (\ref{eq:def2_N_{Phi(I)^*}}) as the definition of the half norm map $N_{\Phi(I)}$ and $N_{\Phi(I)^*}$.
It is well-defined since $\sum_{\varphi H_0 \in I} r_{\Phi}(\sigma)(\varphi H_0)$ depends only on the double coset of $H(I) \sigma \hreflex(\Phi)$.
By this, we can calculate $N_{\Phi(I)}$, $N_{\Phi(I)^*}$ for $I \subset \SetN$ with even cardinality. 
It is also possible to omit the assumption that $\iota$ acts on $M$ as $-1$.

$N_{\Phi(I)}$ gives a map from $M^{H(I)}$ to $M^{\hreflex(\Phi)}$
because we have for $\sigma \in \hreflex(\Phi)$,
\begin{eqnarray*}
	\sigma N_{\Phi(I)} &=& \frac{1}{2} \sum_{\psi \in H(I) \backslash G } (-1)^{\sum_{\varphi H_0 \in I} r_{\Phi}(\psi)(\varphi H_0)} \sigma \psi^{-1} \\
								&=& \frac{1}{2} \sum_{\psi \in H(I) \backslash G } (-1)^{\sum_{\varphi H_0 \in I} r_{\Phi}(\psi \sigma )(\varphi H_0)} \psi^{-1} = N_{\Phi(I)}. \\
\end{eqnarray*}
Similarly, $N_{\Phi(I)^*}(M^{\hreflex(\Phi)}) \subset M^{H(I)}$, 
$\iota N_{\Phi(I)} = (-1)^\abs{I} N_{\Phi(I)}$ and $\iota N_{\Phi(I)^*} = (-1)^\abs{I} N_{\Phi(I)^*}$ also hold.

For a $G$-module $M$, we denote the complex conjugation of $a \in M$ by $\bar{a}$. 
Let $L$ be a CM-field or a totally real field, and $F$ be a subfield of $L$.
Then, a positive definite quadratic form over $L$ is defined by $a \mapsto Tr_{L/F}(\bar{a}a)$.

On each direct sum
$\bigoplus_{\Phi \in \Lambda} \kreflex(\Phi)$ and
$\bigoplus_{I \in \Jodd} K(I)$,
a quadratic form is defined by the orthogonal sum of $Tr_{K(I)/\RationalField}(\bar{a}a)$, $Tr_{\kreflex(\Phi)/\RationalField}(\bar{a}a)$ respectively.
Denote them by $Q_J$ and $Q_\Lambda$.
They have canonical linear extensions to 
			$K_0^c \otimes_\RationalField ( \bigoplus_{I \in \Jodd} K(I) )$,
			$K_0^c \otimes_\RationalField ( \bigoplus_{\Phi \in \Lambda} \kreflex(\Phi) )$,
which are denoted by $K_0^c \otimes_{\RationalField} Q_J$ and $K_0^c \otimes_{\RationalField} Q_\Lambda$.

For the maximum totally real field $K_0$ of $K$,
take a totally positive element $d \in K_0$ such that $K = \QuadExt{K_0}{-d}$.
Then, we have the Pfister form defined over $K_0^c$:
\begin{eqnarray}\label{eq:definition of q}
	q := \langle 1, \varphi_1(d) \rangle \otimes \cdots \otimes \langle 1, \varphi_N(d) \rangle,
\end{eqnarray}
where $\langle 1, a \rangle$ represents the quadratic form $x^2 + a y^2$.
This quadratic form can be regarded as a tensor product of the norm $N_{ \QuadExt{K_0^c}{-\varphi_i(d)} / K_0^c }$ 
defined on 
	\begin{eqnarray}\label{eq:definition of V_q}
		V := \QuadExt{K_0^c}{-\varphi_1(d)} \otimes_{K_0^c} \cdots \otimes_{K_0^c}  \QuadExt{K_0^c}{-\varphi_N(d)},
	\end{eqnarray}
which is a $K_0^c$-algebra. 

For $I \subset \SetN$, let
\begin{eqnarray}
	v_I := c_1 \otimes \cdots \otimes c_N,\ 
	c_i := \begin{cases}
				\sqrt{ -\varphi_i(d) } & \text{if } \varphi_i H_0 \in I,  \\
				1  &  \text{otherwise.}
			\end{cases}%
\end{eqnarray}
Then, $\{ v_I : I \subset \SetN \}$ makes a basis of $V$ as a linear space over $K_0^c$.

Now, we have $\prod_{\varphi H_0 \in I} \sqrt{ -\varphi (d) }  \in K(I)$, since $H(I)$ fixes it.  
Let $[I] := \{ \sigma I : \sigma \in G \}$.
Then, there is the canonical isomorphism:
\begin{eqnarray}
	K_0^c \otimes_\RationalField K(I) \stackrel{\cong}{\longrightarrow } \bigoplus_{I^\prime \in [I]} \FieldExt{K_0^c}{ \textstyle \prod_{\varphi \in I^\prime} \sqrt{ -\varphi (d) } }.
\end{eqnarray}
We denote by $J$, a system of representatives for the orbits of the action of $G$ on $\SetN$.
Then,
\begin{eqnarray}
	K_0^c \otimes_\RationalField \left( \textstyle \bigoplus_{I \in J} K(I) \right) \cong \bigoplus_{I \subset \SetN} \FieldExt{K_0^c}{ \textstyle \prod_{\varphi \in I} \sqrt{ -\varphi (d) } }.
\end{eqnarray}
This map gives an embedding
	$\phi_{\Jodd} : K_0^c \otimes_\RationalField ( \bigoplus_{I \in \Jodd} K(I)^{1-\iota} ) \hookrightarrow V$.
Hence, by (\ref{eq:eq3}), there is also an embedding $\phi_{\Jodd} \circ N_{\Lambda \rightarrow J} : K_0^c \otimes_\RationalField ( \bigoplus_{\Phi \in \Lambda} \kreflex(\Phi)^{1-\iota} ) \hookrightarrow V$.

Theorem \ref{thm:Pfister form and reflex fields} shows that $\phi_{\Jodd} \circ N_{\Lambda \rightarrow J}$ is extended to an isomorphism between 
$K_0^c \otimes_\RationalField ( \bigoplus_{\Phi \in \Lambda} \kreflex(\Phi) )$ and $V$,
conserving their quadratic forms.

\begin{thm} \label{thm:Pfister form and reflex fields}
For a totally positive $d \in K_0$ such that $K = K_0(\sqrt{-d})$, 
define a Pfister form $q$ and $K_0^c$-algebra $V$ by (\ref{eq:definition of q}) and (\ref{eq:definition of V_q}) respectively.
Then, 
there is an isomorphism as $K_0^c$-algebras:
\begin{eqnarray*}
	\phi_\Lambda : K_0^c \otimes_\RationalField \left( \bigoplus_{\Phi \in \Lambda} \kreflex(\Phi) \right) \stackrel{\cong}{\longrightarrow} V.
\end{eqnarray*}
Furthermore, such isomorphisms satisfy $q( \phi_\Lambda(w) ) = 2^{-N} K^c_0 \otimes_{\RationalField} Q_\Lambda(w)$,
where $Q_\Lambda$ is the orthogonal sum of the quadratic form $Tr_{\kreflex(\Phi)/\RationalField}(\bar{a}a)$.
\end{thm}
The following lemma is used for the proof.
\begin{lem}\label{lem:eq3}
	Let $I, I^\prime$ be subsets of $\SetN$, and denote their exclusive disjunction by $I \veebar I^\prime$.
	For any $2$-divisible $G$-module $M$ and $\Phi_f, \Phi_{f^\prime} \in \Lambda$, let $a, b$ be elements of $M$ fixed by $\hreflex(\Phi_f), \hreflex(\Phi_{f^\prime})$ respectively. Then,
	\begin{eqnarray}\label{eq:lemma_eq3}
			& & \hspace{-10mm}
				\sum_{I \subset \SetN}
					(-1)^{\sum_{\varphi H_0 \in I} f(\varphi H_0) + \sum_{\varphi H_0 \in I \veebar I^\prime} f^\prime(\varphi H_0)}
							N_{\Phi_f(I)^*}(a)
								N_{\Phi_{f^\prime}(I \veebar I^\prime)^*}(b)
							\nonumber \\
				\hspace{10mm}
				&=&
					\begin{cases}
						\displaystyle 2^{N-1}
							(-1)^{\sum_{\varphi H_0 \in I^\prime} f(\varphi H_0)}
									N_{\Phi_f(I^\prime)^*}(a \times \sigma(b) )  & \text{if } f = \sigma * f^\prime \text{ for some } \sigma \in G,  \\
						0  &  \text{otherwise.}\\
					\end{cases} \hspace{10mm}%
	\end{eqnarray}
\end{lem}
%
%
\begin{proof}
	The left-hand side of (\ref{eq:lemma_eq3}) equals
	\begin{eqnarray*}
			& & \hspace{-10mm}
			\frac{1}{4}	
				\displaystyle \sum_{\psi \in G/\hreflex(\Phi_f)}
					\displaystyle \sum_{\psi^\prime \in G/\hreflex(\Phi_{f^\prime}) }
							\psi (a)
								\psi^\prime (b) \\
				& & \times				
									\displaystyle \sum_{ I \subset \SetN }
									(-1)^{\sum_{\varphi H_0 \in I} ( f + r_{\Phi_f}(\psi) )(\varphi H_0) 
												+ \sum_{\varphi H_0 \in I \veebar I^\prime } ( f^{\prime} + r_{\Phi_{f^\prime}}(\psi^\prime) )(\varphi H_0)}
				=: (\star).
	\end{eqnarray*}
	We have
	\begin{eqnarray} \label{eq:tranform1}
			& & \hspace{-10mm}
	\sum_{\varphi H_0 \in I } ( f + r_{\Phi_f}(\psi) )(\varphi H_0) + \sum_{\varphi H_0 \in I \veebar I^\prime } (f^\prime + r_{\Phi_{f^\prime}}(\psi^\prime) )(\varphi H_0)  \nonumber \\
			&=& \sum_{\varphi H_0 \in I } ( f + r_{\Phi_f}(\psi) )(\varphi H_0) + \sum_{\varphi H_0 \in I^\prime } ( f + r_{\Phi_f}(\psi) )(\varphi H_0) \nonumber \\
			& &		+ \sum_{\varphi H_0 \in I^\prime } ( f + r_{\Phi_f}(\psi) )(\varphi H_0) + \sum_{\varphi H_0 \in I \veebar I^\prime } (f^\prime + r_{\Phi_{f^\prime}}(\psi^\prime) )(\varphi H_0)  \nonumber \\
			&=& \sum_{\varphi H_0 \in I \veebar I^\prime } ( f + f^\prime + r_{\Phi_f}(\psi) + r_{\Phi_{f^\prime}}(\psi^\prime) ) (\varphi H_0)
					+ \sum_{\varphi H_0 \in I^\prime } ( f + r_{\Phi_f}(\psi) )(\varphi H_0). \hspace{10mm}
	\end{eqnarray}
	Furthermore, since $r_{ \Phi_{f^\prime} }(\tau) = r_{\Phi_f}(\tau) + \tau \cdot (f^\prime - f) - (f^\prime - f)$, we have 
	\begin{eqnarray*}
		 f + f^{\prime} + r_{\Phi_f}(\psi) + r_{\Phi_{f^\prime}}(\psi^\prime)
			&=& \psi \cdot( f - f^{\prime} )  + r_{\Phi_{f^\prime}}(\psi) + r_{\Phi_{f^\prime}}(\psi^\prime) \\
				&=& \psi \cdot( f - f^{\prime} )  + r_{\Phi_{f^\prime}}(\psi) + r_{\Phi_{f^\prime}}(\psi) + \psi \cdot r_{\Phi_{f^\prime}}(\psi^{-1} \psi^\prime)  \\
				&=& \psi \cdot( f - (\psi^{-1} \psi^\prime) * f^{\prime} ).
	\end{eqnarray*}
	Hence, 
	\begin{eqnarray*}
			(\star) 
			&=&
			\frac{1}{4}	
				\displaystyle \sum_{\psi \in G/\hreflex(\Phi_f)}
					\displaystyle \sum_{\psi^\prime \in G/\hreflex(\Phi_{f^\prime}) }
							(-1)^{\sum_{\varphi H_0 \in I^\prime } ( f + r_{\Phi_f}(\psi) )(\varphi H_0)}
							\psi (a)
								\psi^\prime (b) \\
			& & \times \displaystyle \sum_{ I \subset \SetN }
									(-1)^{\sum_{\varphi H_0 \in I \veebar I^\prime} \psi \cdot \left(f - (\psi^{-1} \psi^\prime) * f^{\prime} \right)(\varphi H_0)}.
	\end{eqnarray*}
	Hence, 
	let $\sigma = \psi^{-1} \psi^\prime$, then,
	\begin{eqnarray*}
			(\star) 
			&=&
				\begin{cases}
					\displaystyle 2^{N-2} \sum_{\psi \in G / \hreflex(\Phi_f)} (-1)^{\sum_{\varphi H_0 \in I^\prime} (f + r_{\Phi_f}(\psi))(\varphi H_0)} \psi(a \times \sigma(b) )  & \text{if } f = \sigma * f^\prime \text{ for some } \sigma \in G,  \\
					0  &  \text{otherwise.}
				\end{cases}%
	\end{eqnarray*}
	Therefore, the equation (\ref{eq:lemma_eq3}) follows.
\end{proof}

\begin{proof}[Proof of Theorem \ref{thm:Pfister form and reflex fields}]
If there is an isomorphism $\phi_\Lambda$ such that that $q( \phi_\Lambda(w) ) = 2^{-N} Q_\Lambda(w)$, 
the other isomorphisms also have the property, because $Q_\Lambda \circ \sigma= Q_\Lambda$ holds for any automorphisms $\sigma$ 
of $K_0^c \otimes_\RationalField ( \bigoplus_{\Phi \in \Lambda} \kreflex(\Phi) )$.

For $(a_{\Phi_f}) \in \bigoplus_{\Phi_f \in \Lambda} \kreflex(\Phi_f)$, define $\phi_\Lambda$ by
\begin{eqnarray}\label{eq:definition of phi_Lambda}
	\phi_\Lambda((a_{\Phi_f}))
			:= 2^{-N+1} \sum_{ I }   
							\frac{ v_I }{ \prod_{\varphi H_0 \in I} \sqrt{-\varphi(d)} } \sum_{\Phi_f \in \Lambda} (-1)^{ \sum_{\varphi H_0 \in I } f(\varphi H_0) }
								N_{\Phi_f(I)^*}(a_{\Phi_f}). 
\end{eqnarray}
The coefficient of $v_i$ is an element of $K_0^c$,
because we have $c N_{\Phi_f(I)^*} = (-1)^{ \sum_{\varphi H_0 \in I } r_{\Phi_f}(c)(\varphi H_0) }  N_{\Phi_f(I)^*}$ for $c \in C$, 
and furthermore, by Lemma \ref{lem:embedding r} and Proposition \ref{prop:1-cocycle r},
\begin{eqnarray}
	c( \prod_{\varphi H_0 \in I} \sqrt{-\varphi(d)} ) = (-1)^{ \sum_{\varphi H_0 \in I } r_{\Phi_f}(c)(\varphi H_0) }  \prod_{\varphi H_0 \in I} \sqrt{-\varphi(d)}.
\end{eqnarray}
Hence, (\ref{eq:definition of phi_Lambda}) is well-defined.

Let's see that $\phi_\Lambda$ satisfies the assertion of the theorem.
It is clear that $\phi_\Lambda$ is $K_0^c$-linear.
We divide the proof into two parts.

\begin{enumerate}[(i)]
\item $\phi_\Lambda$ is a homomorphism between $K_0^c$-algebras;
it is enough if we can show $\phi_\Lambda((a_\Phi)) \phi_\Lambda((b_\Phi)) = \phi_\Lambda((a_\Phi b_\Phi))$ 
for any $(a_\Phi)_{\Phi \in \Lambda}, (b_\Phi)_{\Phi \in \Lambda} \in \bigoplus_{\Phi \in \Lambda} \kreflex(\Phi)$.
Denote the exclusive disjunction of  $I$ and $I^\prime$ by $I \veebar I^\prime$.
Then, we have
\begin{small}
\begin{eqnarray*}
	\phi_\Lambda  ( (a_{\Phi_f}) ) 
	\phi_\Lambda  ( (b_{\Phi_f}) ) 
			&=& 2^{-2 N}  \sum_{\Phi_f \in \Lambda} \sum_{ \Phi_{f^\prime} \in \Lambda} 
					\sum_{I}
					\sum_{I^\prime} 
							\frac{ v_{I \veebar I^\prime}  }{ \prod_{\varphi H_0 \in I \veebar I^\prime} \sqrt{-\varphi(d)} }  \\
			& & \times
						\sum_{\psi \in G / \hreflex(\Phi_f)} (-1)^{ \sum_{\varphi H_0 \in I } (f + r_{\Phi_f}(\psi))(\varphi H_0) }
							\psi( a_{\Phi_f} ) \\
			& & \times
						\sum_{\psi^\prime \in G / \hreflex(\Phi_{f^\prime})} (-1)^{ \sum_{\varphi H_0 \in I^\prime } (f^\prime + r_{\Phi_{f^\prime}}(\psi^\prime))(\varphi H_0) }
							\psi^\prime( b_{\Phi_{f^\prime}} ) \\
			&=& 2^{-2 N}  \sum_{\Phi_f \in \Lambda} \sum_{ \Phi_{f^\prime} \in \Lambda} 
					\sum_{I}
					\sum_{I^\prime} 
							\frac{ v_{I^\prime}  }{ \prod_{\varphi H_0 \in I^\prime} \sqrt{-\varphi(d)} }  \\
			& & \times
						\sum_{\psi \in G / \hreflex(\Phi_f)} (-1)^{ \sum_{\varphi H_0 \in I } (f + r_{\Phi_f}(\psi))(\varphi H_0) }
							\psi (a_{\Phi_f} ) \\
			& & \times
						\sum_{\psi^\prime \in G / \hreflex(\Phi_{f^\prime})} (-1)^{ \sum_{\varphi H_0 \in {I \veebar I^\prime} } (f^\prime + r_{\Phi_{f^\prime}}(\psi^\prime))(\varphi H_0) }
							\psi^\prime (b_{\Phi_{f^\prime}} ) \\
			&=& 2^{-2 N+2}  \sum_{\Phi_f \in \Lambda} \sum_{ \Phi_{f^\prime} \in \Lambda}
					\sum_{I^\prime} 
							\frac{ v_{I^\prime}  }{ \prod_{\varphi H_0 \in I^\prime} \sqrt{-\varphi(d)} } \\
				& & \times
					\sum_{I}
						(-1)^{ \sum_{\varphi H_0 \in I } f(\varphi H_0) + \sum_{\varphi H_0 \in {I \veebar I^\prime} } f^\prime(\varphi H_0)}
							N_{\Phi_f(I)^*}(a_{\Phi_f} ) N_{ \Phi_{f^\prime}(I)^* }(b_{\Phi_{f^\prime}} ).
\end{eqnarray*}
\end{small}
Therefore, by Lemma \ref{lem:eq3}, 
\begin{eqnarray*}
	\phi_\Lambda (a_{\Phi_f}) 
	\phi_\Lambda  (b_{\Phi_f} )
			&=& 2^{-N+1} \sum_{\Phi_f \in \Lambda} 
					\sum_{I^\prime}
							\frac{ v_{I^\prime} }{ \prod_{\varphi H_0 \in I^\prime} \sqrt{-\varphi(d)} }
									(-1)^{\sum_{\varphi H_0 \in I^\prime } f(\varphi H_0) }  N_{\Phi_f(I^\prime)^*}(a_{\Phi_f}  b_{\Phi_f})  \\
			&=&  \phi_\Lambda (a_{\Phi_f} b_{\Phi_f}).
\end{eqnarray*}

\item $q( \phi_\Lambda(w) ) = 2^{-N} K_0^c \otimes_\RationalField Q_\Lambda(w)$;
it is enough if we can show $q( \phi_\Lambda( (a_\Phi) ) ) = 2^{-N} Q_\Lambda( (a_\Phi) )$ for $(a_\Phi)_{\Phi \in \Lambda} \in \bigoplus_{\Phi \in \Lambda} \kreflex(\Phi)$.
We have 
\begin{eqnarray*}
	q( \phi_\Lambda((a_{\Phi_f})) )
			&=& 2^{-2 N+2}
					\sum_{I} 
							\frac{ q(v_I) }{ \prod_{\varphi H_0 \in I} ( -\varphi(d) ) } 
				\left(
					\sum_{\Phi_f \in \Lambda}
							(-1)^{ \sum_{\varphi H_0 \in I } f(\varphi H_0) } N_{\Phi_f(I)^*}(a_{\Phi_f})
				\right)^2 \\
			&=& 2^{-2 N+1}
				\sum_{I} (-1)^{ \abs{I} }
					\sum_{\Phi_f \in \Lambda} 
						(-1)^{ \sum_{\varphi H_0 \in I } f(\varphi H_0) } N_{\Phi_f(I)^*}(a_{\Phi_f}) \\
			& & \times
					\sum_{\Phi_{f^\prime} \in \Lambda}
						\sum_{\psi^\prime \in G / \hreflex(\Phi_{f^\prime})}
							(-1)^{ \sum_{\varphi H_0 \in I } ({f^\prime} + r_{\Phi_{f^\prime}}(\psi))(\varphi H_0) }
										\psi^\prime ( a_{\Phi_{f^\prime}} ).
\end{eqnarray*}
Since $\iota * f(\varphi H_0) = ( r_{\Phi_0}(\iota) + \iota \cdot f )(\varphi H_0) = f(\varphi H_0) +1$ for any $f$,
\begin{eqnarray*} 
	q( \phi_\Lambda((a_{\Phi_f})) )
			&=& 2^{-2 N+1}
				\sum_{I} 
					\sum_{\Phi_f \in \Lambda} (-1)^{ \sum_{\varphi H_0 \in I } f(\varphi H_0) } N_{\Phi_f(I)^*}(a_{\Phi_f}) \\
			& & \times
					\sum_{\Phi_{f^\prime} \in \Lambda}
						\sum_{\psi^\prime \in G / \hreflex(\Phi_{f^\prime})}
						(-1)^{ \sum_{\varphi H_0 \in I } ({ \iota * f^\prime} + r_{\Phi_{ \iota * f^\prime }}(\psi))(\varphi H_0) }
										\psi^\prime ( a_{\Phi_{f^\prime}} )   \\
			&=& 2^{-2 N+2}
				\sum_{I} 
					\sum_{\Phi_f \in \Lambda}
					\sum_{\Phi_{f^\prime} \in \Lambda}
						(-1)^{ \sum_{\varphi H_0 \in I } f(\varphi H_0) + \sum_{\varphi H_0 \in I } \iota * f^\prime(\varphi H_0) } \\
			& & \times
							N_{\Phi_f(I)^*}(a_{\Phi_f}) N_{\Phi_{\iota * f^\prime}(I)^*}(a_{\Phi_f}).
\end{eqnarray*}

Therefore,
by Lemma \ref{lem:eq3}, 
\begin{eqnarray*}
	q( \phi_\Lambda((a_{\Phi_f})) )
			&=& 2^{-N+1}
					\sum_{\Phi_f \in \Lambda}
						N_{ \Phi_f(\emptyset)^* }( a_{\Phi_f} \overline{a_{\Phi_f}} )  \\
			&=& 2^{-N}
					\sum_{\Phi_f \in \Lambda}
						\sum_{\psi \in G / \hreflex(\Phi_f)} 
										\psi ( a_{\Phi_f} \overline{a_{\Phi_f}} )  
			= 2^{-N} Q_{\Lambda}((a_{\Phi_f})).
\end{eqnarray*}
\end{enumerate} 

It is clear that $q$ and  $Q_\Lambda(w)$ are positive definite.
Since $q( \phi_\Lambda(w) ) = 2^{-N} K_0^c \otimes_\RationalField Q_\Lambda(w)$ for any $w \in K_0^c \otimes_\RationalField ( \bigoplus_{\Phi \in \Lambda} \kreflex(\Phi) )$, $\phi_\Lambda$ must be injective.
Hence, by comparing the dimensions, it is an isomorphism.
Therefore, the assertion holds.
\end{proof}

\section*{Acknowledgements}
\label{}
I want to thank Professor Takayuki Oda for his teaching and advice in my graduate school days.
I also appreciate Professor Hiromichi Yanai for his helpful comments. 




\end{document}